\documentclass[a4paper,12pt]{amsart}
\usepackage{paralist} 
\usepackage[T1]{fontenc}
\usepackage{currvita}
\usepackage{geometry}
\usepackage{amsmath}
\allowdisplaybreaks
\usepackage{amssymb}
\usepackage{amsfonts}
\usepackage{color}
\usepackage{enumerate}

\usepackage{booktabs}

\usepackage{etoolbox}
\patchcmd{\thebibliography}{\section*{\refname}}{}{}{}

\usepackage{verbatim}

\usepackage{graphicx}
\usepackage{subfig}
\usepackage{bm}
\usepackage{tikz}
\bgroup

\numberwithin{equation}{section}

\usepackage[bookmarks,bookmarksopen=false,
                  pdfauthor={Autor},
                  pdftitle={Titel},
                  colorlinks=true,
                  linkcolor=blue,
                  citecolor=blue,
                  urlcolor=blue,
                ]{hyperref}

\newtheorem{lemma}{Lemma}[section]
\newtheorem{theorem}[lemma]{Theorem}

\newcommand{\pbox}{\hfill$\Box$\\}

\newcommand{\R}{\mathbb{R}}

\newcommand{\Z}{\mathbb{Z}}

\newcommand{\cG}{\mathcal{G}}

\newcommand{\W}{\mathcal{W}}

\newcommand{\cS}{\mathcal{S}}
\newcommand{\cU}{\mathcal{U}}

\renewcommand{\H}{\mathcal{H}}

\newcommand{\inv}{^{-1}}
\def\rd{\R^d}
\def\rdd{{\R^{2d}}}

\def\lrd{L^2(\rd)}

\newcommand{\modsp}{modulation space}
\newcommand{\tfa}{time-frequency analysis}
\newcommand{\stft}{short-time Fourier transform}
\newcommand{\tf}{time-frequency}
\newcommand{\tfs}{time-frequency shift}

\newcommand{\norm}[2]{\left\| #2 \right\|_{#1}}

\definecolor{darkviolet}{rgb}{0.58,0,0.83}
\usetikzlibrary{patterns}

\begin{document}
\begin{abstract}
We study one-parameter families of pseudodifferential operators whose Weyl symbols are obtained by dilation and a smooth deformation of a symbol in a weighted Sj{\"o}strand class. We show that their spectral edges are Lipschitz continuous functions of the dilation or deformation parameter. Suitably local estimates hold also for the edges of every spectral gap.
These statements extend Bellissard's seminal results on the Lipschitz
continuity of spectral edges for families of operators with periodic
symbols to a large class of symbols with only mild regularity
assumptions. 

The abstract results are used to prove that the frame bounds of a family of Gabor systems $\mathcal{G}(g,\alpha\Lambda)$, where $\Lambda$ is a set of non-uniform time-frequency shifts, $\alpha>0$, and $g\in M^1_2(\R^d)$, are  Lipschitz continuous functions in $\alpha$. This settles a question about the precise blow-up rate of the condition number of Gabor frames near the critical density.
\end{abstract}

\title[Lipschitz continuity of spectra]{Lipschitz Continuity of Spectra of Pseudodifferential Operators
in a Weighted Sj\"ostrand Class and Gabor Frame Bounds}

\author[K. Gr\"ochenig]{Karlheinz Gr\"ochenig}
\address[K. G.]{Faculty of Mathematics \\
	University of Vienna \\
	Oskar-Morgenstern-Platz 1 \\
	A-1090 Vienna, Austria }
\email{karlheinz.groechenig@univie.ac.at}

\author[J. L. Romero]{Jos\'e Luis Romero}
\address[J. L. R.]{Faculty of Mathematics \\
	University of Vienna \\
	Oskar-Morgenstern-Platz 1 \\
	A-1090 Vienna, Austria \\and
	Acoustics Research Institute\\ Austrian Academy of Sciences\\Wohllebengasse 12-14, Vienna, 1040, Austria}
\email{jose.luis.romero@univie.ac.at}

\author[M. Speckbacher]{Michael Speckbacher}

\address[M. S.]{Faculty of Mathematics \\
	University of Vienna \\
	Oskar-Morgenstern-Platz 1 \\
	A-1090 Vienna, Austria }
\email{michael.speckbacher@univie.ac.at}

\date{}
\thanks{The authors gratefully acknowledge the support of the
Austrian Science Fund (FWF) through the projects   P31887-N32 (K.G.) and Y1199 (J.L.R. and M.S.).}
\maketitle

\section{Introduction}

We consider a one-parameter family of operators $T_\delta $ that
depend smoothly on $\delta $. A basic problem of spectral theory is to
understand how the spectrum $\sigma (T_\delta )$ depends on the
parameter. For self-adjoint operators at least, one would expect that
the spectrum depends continuously, e.g., in the Hausdorff metric, on
$\delta $. This has been shown in several general
settings~\cite{atmapu10,avmosi90,ell82}. As to more quantitative results, one may
investigate the behavior of the extreme spectral values, or more
generally of spectral edges, and try to understand their smoothness as
a function of $\delta $. This is an interesting problem in
mathematical physics where $\delta $ may be the magnitude of a
magnetic field or the value of Planck's constant~\cite{atmapu10,bel94,co10, copu12, copu15,bebe16,bebeni18,beta21}. 

In this paper we study a general class of pseudodifferential operators
with symbols in a weighted Sj\"ostrand class and $\delta $
amounts roughly to a dilation of the symbol. We will prove that the
spectral edges are Lipschitz continuous in $\delta$. Moreover, we
will show the Lipschitz continuity of the spectral gaps.

In our study we use the Weyl calculus (though other calculi work as
well without essential changes). 
Given  a symbol $\sigma\in \mathcal{S}'(\R^{2d})$, its Weyl transform
is the operator 
\begin{equation}
  \label{eq:c0}
\sigma^w
f(y)=\int_{\R^{2d}}\sigma\left(\frac{x+y}{2},\omega\right)e^{2\pi i
  (y-x)\cdot \omega}f(x)dxd\omega ,   
\end{equation}
for $f\in \cS (\rd )$ and  a suitable interpretation of the integral. 
Let $D_a $ denote the dilation  $D_a \sigma (z) = \sigma (az)$. 
Throughout we will study a one-parameter family of symbols
$\sigma_\delta $ that arises by a dilation and a smooth variation of a
basic symbol as follows. We will assume that the symbol depends on a parameter
$\delta\in(-\delta_0,\delta_0)$  
like 
$$
\sigma_\delta=D_{\sqrt{1+\delta}}G_\delta,
$$
and  write $T_\delta:=\sigma_\delta^w$ for the corresponding operators. While $G_\delta$ is allowed to vary with $\delta$ we shall assume that this dependence is moderate, so that $\sigma_\delta$ is roughly a dilation. For the question of spectral perturbation to be meaningful, we will
assume that $\sigma _\delta $ is real-valued and that the
corresponding operator is bounded on $\lrd $, whence $T_\delta $ is
always self-adjoint.

Our questions are thus: How does the spectrum of $T_\delta $
depend on $\delta $? Consider in particular the spectral extreme
values  $\sigma_-(A):=\inf\{\lambda\in\R:\ \lambda\in\sigma(A)\}$ and
$\sigma_+(A):=\sup\{\lambda\in\R:\ \lambda\in\sigma(A)\}$ of a 
self-adjoint operator. How does $\sigma _{\pm}  (T_\delta )$ depend on $\delta
$? Or more generally, how do the spectral edges - that is, the endpoints of the connected components of $\mathbb{R} \setminus \sigma(T_\delta)$ - depend on $\delta$?  What are suitable conditions on the symbols $\sigma _\delta $ so that the spectral edges are Lipschitz continuous?

Our main inspiration comes from   J.\
Bellissard's fundamental paper ~\cite{bel94} on  the  almost-Matthieu operator or Harper
operator in a non-commutative torus. He  showed that, for certain families
of  Harper-like operators on the square lattice with constant magnetic
field, the spectral edges and in particular the spectral gap boundaries depend Lipschitz
continuously on the  parameter, improving previous  results on
(H\"older-)continuity of spectral  gaps \cite{ell82} and   spectral
edges \cite{avmosi90}.  Bellissard's work has inspired
many authors to extend his results.   
  In \cite{ko03}, for example,  Lipschitz continuity for
  Harper-like operators on crystal lattices is shown, while
  \cite{co10,cohepu21,copu12,copu15} considered continuous magnetic Schr\"odinger
  operators with weak magnetic field perturbation, \cite{atmapu10} showed spectral continuity of
  pseudodifferential operators with elliptic symbols in the H\"ormander class, and \cite{beta21} studied dynamically-defined operator families on groups of polynomial growth, and the Lipschitz continuity of their spectra. The methods from \cite{bel94} have also proved useful to investigate fine properties of rotation algebras \cite{gesha22}. Most notably, Beckus and Bellissard \cite{bebe16} have distilled part of the argument of \cite{bel94} into a set of powerful abstract principles (see Section \ref{sec_m}).

\subsection{Results}

To formulate our results, we use the language and methods  of phase-space
analysis (\tfa\ in applied mathematics) and employ a class of symbols that is tailored to \tfa . Let
$z= (x,\omega )\in \rdd $ be a point in phase-space (\tf\ space), and 
$$
\rho(z)f(t) =M_{\omega/2}T_x M_{\omega/2}f(t) = e^{-i\pi x\cdot
  \omega} e^{2\pi i \omega \cdot t} f(t-x)
$$ denote the (symmetric) time-frequency shift by $z$. The associated
transform is the \stft\
$$
V_gf(z) = e^{-i\pi x\cdot \omega}\langle f, \rho (z) g \rangle =    \int _{\rd } f(t) \overline{g(t-x)}
e^{-2\pi i \omega \cdot t} \, dt \,  .
$$
 Let $\varphi(t)=2^{d/4}e^{-\pi |t|^2}$ denote the standard Gaussian
 in $\R^d$. The \emph{mixed-norm weighted modulation} space
 $M^{p,q}_{s,t}(\R^d)$, $1\leq p,q\leq\infty,$ $s,t\geq0,$ contains
 all the elements in $\mathcal{S}'(\R^d)$ for which  the norm 
\begin{equation}\label{def:modul}
\| f\|_{M^{p,q}_{s,t}}:=\left(\int_{\R^{d}}\left(\int_{\R^{d}}|V_\varphi f(x,\omega)|^p(1+|x|)^{sp}dx\right)^{q/p}(1+|\omega|)^{tq} d\omega\right)^{1/q} 
\end{equation}
 is finite, with the usual modification when $p=\infty $ or
 $q=\infty$.  If $p=q$, we write $M^p_{s,t}(\R^d)$, and if  $s=t$, we
 write $M^{p,q}_{s}(\R^d)$.  We omit the subscripts, when $s=t=0$ and
 write $M^{p,q} $ for $M^{p,q}_{0,0}$. 
The use of any nonzero function $g\in \mathcal{S}(\R^d)$ instead of the Gaussian in \eqref{def:modul} gives an equivalent norm, i.e. $\| f\|_{M^{p,q}_{s,t}}\asymp \|V_gf\|_{L^{p,q}_{s,t}}$. See, e.g., \cite{groe1} for more details.

  Our main result reads as follows.
\begin{theorem}\label{thm:main-general}
Let $0<\delta_0<1$. For $\delta \in (-\delta_0,\delta_0)$, let
$G_\delta \in M^{\infty,1}_{0,2}(\R^{2d})$ be real-valued and $\delta\mapsto G_\delta  $ differentiable\footnote{Here we mean that the partial derivative of $(z,\delta)\mapsto G_\delta(z)$ with respect to $\delta$ exists.} such  that  $\partial_\delta
{G_\delta}\in M^{\infty,1}(\R^{2d})$.  Let $T_\delta = \sigma _\delta ^w $
be the pseudodifferential operator with Weyl symbol $\sigma _\delta =
D_{\sqrt{1+\delta }} G_\delta $.   Then, for
$\delta_1,\delta_2 \in (-\delta _0, \delta _0)$,
\begin{equation*}
|\sigma_\pm(T_{\delta_1} )-\sigma_\pm(T_{\delta_2})|\leq C_d \cdot |\delta_1-\delta_2|\cdot (1-\delta_0)^{-(d+1)} \cdot \sup_{|t|<\delta_0} \big( \left\|  G_t\right\|_{M^{\infty,1}_{0,2}} +\left\| \partial_tG_t\right\|_{M^{\infty,1}}  \big),
\end{equation*}
where $C_d$ is a constant that only depends on $d$.
\end{theorem}
To put this result in perspective, let us compare it with that in \cite{bel94}. Bellissard studied operators
that are linear combinations of phase-space shifts over a lattice,
i.e., operators of the form 
\begin{equation}\label{eq:bels-op}
 T_\delta  = \sum _{k\in \Z ^2 } a_k (\delta ) \rho
\big(\sqrt{1 +\delta} \, k \big).
\end{equation}
Then  the Weyl symbol of $T_\delta $  is periodic and  given by 
\begin{align}\label{eq_ws}
\sigma _\delta (z) = \sum _{k\in \Z ^2 } a_k (\delta )
e^{2\pi i \sqrt{1 + \delta } \,  [k,z]
} \, ,
\end{align}
where   $[z,z']=x'\cdot \omega-x\cdot \omega'$  denotes the standard symplectic form.
Bellissard's condition on the coefficients $a_k(\delta)$ explicitly reads
$$
\sup_{|t|<\delta_0}\left(\sum_{k\in\Z^2}|a_k(t)|^2(1+|k|)^{6+2\varepsilon}+\sum_{k\in\Z^2}|\partial_t a_k(t)|^2(1+|k|)^{2+2\varepsilon}\right)<\infty,
 $$
for some $\varepsilon>0.$
Since $\big\|\sum_{k\in\Z^2}b_k e^{2\pi i[k,\cdot]}\big\|_{M^{\infty,1}_{0,s} }^2\lesssim \sum_{k\in\Z^2}|b_k|^2(1+|k|)^{2(1+s+\varepsilon)}$, it follows that our result extends Bellissard's original conditions.

An important point is that operators of the form \eqref{eq:bels-op} belong to the
non-commutative torus based on $\sqrt{1+ \delta } \, \Z ^2 $, and thus
$C^*$-algebraic arguments may be applied. By contrast, Theorem~\ref{thm:main-general}
uses non-periodic symbols within the class $M^{\infty,1}_{0,2}(\R^{2d})$, which roughly corresponds to requiring two bounded derivatives, although the precise membership condition is slightly more subtle. In particular, the class $M^{\infty,1}_{0,2}(\R^{2d})$  
contains the H\"ormander class $S^{0}_{0,0}$  of infinitely smooth
symbols. 
 The  Sj\"ostrand class $M^{\infty ,1}_{0,2}(\R^{2d})$ is perhaps not as
known as the H\"ormander classes, but it has become a common and
very natural class of symbols whenever operators  are
defined via phase-space shifts. Indeed, since every pseudodifferential
operator $\sigma ^w$ can be formally represented as a superposition of
phase-space shifts via
$$
\sigma^w=\int_{\R^{2d}}\mathcal{U}\widehat{\sigma}(z)\rho(z)dz,
$$
where $\mathcal{U}F(x,\omega)=F(\omega,-x)$, the appearance of the
Sj\"ostrand class is almost  inevitable. The Sj\"ostrand
class was studied by Sj\"ostrand in ~\cite{Sjo95}, the \tfa\ of
$M^{\infty ,1}_{s,t}(\R^{2d})$ has its origins in~\cite{Gro06}. As it turns out,
$M^{\infty ,1}_{s,t}(\R^{2d})$ is part of a larger family of function spaces, the
\emph{\modsp s}~\cite{fei83}, which have become an indispensable tool in \tfa\ and
the analysis of pseudodifferential operators. For a survey of this
active field we refer to the recent monographs~\cite{BO20,CRbook}.

To match Bellissard's results for the case of pseudodifferential
operators, we will also derive a variant concerning the Lipschitz continuity of spectral gaps. If $g$ is a gap of the spectrum of $A$, that is, a connected component of the resolvent set
$\R \setminus \sigma(A)$, we write
$\sigma_+^g(A)$ and $\sigma_-^g(A )$ to denote  the edges of $g$. 
\begin{theorem} \label{thm:gap-bds}
Under the assumptions of Theorem \ref{thm:main-general}, let $g$ be a gap of the spectrum of $T_0$ with length $L(g)$. Then there exist
$\varepsilon=\varepsilon(g)>0$ and functions $E_+^g,E_-^g:(-\varepsilon,\varepsilon)\rightarrow \R$ such that
for $|\delta| < \varepsilon$:
\begin{itemize}\item[(i)] $E_+^g(0)=\sigma_+^g(T_0)$, and $E_-^g(0)=\sigma_-^g(T_0)$,
\item[(ii)] $E_+^g(\delta)$ (resp. $E_-^g(\delta)$) is the right (resp. left) edge of a gap of  $\sigma(T_\delta)$,   and 
\item[(iii)] (Lipschitz continuity of spectral edges): 
\begin{align*}
&\big|E_\pm^g(\delta )-E_\pm^g(0)\big|
\\
&\quad\leq C_d  \cdot|\delta|\cdot L(g)^{-1}\cdot \sup_{|t|<\delta_0}\left(\|G_t\|_{M^{\infty,1}}\|\partial_tG_t\|_{M^{\infty,1}}+\|G_t\|_{M^{\infty,1}_{0,2}}^2\right).
\end{align*}
\end{itemize}
\end{theorem}
Note that the Lipschitz estimate in Theorem \ref{thm:gap-bds} holds
only for $|\delta|<\varepsilon$. In fact, for larger $\delta$ the gap
may disappear, as it occurs for example in certain graphene-like
systems submitted to constant magnetic fields
\cite{cohepu21}~\footnote{We thank H.\ Cornean for pointing this out
  to us.}.

For periodic Weyl symbols Bellissard \cite{bel94} proved that \[|E_\pm^g(\delta)-E_\pm^g(0)| \leq C_T \cdot  |\delta|\cdot L(g)^{-5},
\qquad |\delta | < \epsilon,\]
and suggested that the correct dependence on the width $L(g)$ of the gap should be $L(g)\inv $. This conjecture was confirmed in \cite[Theorem 4]{bebe16}; see also
\cite[Lemma 10]{bebe16}. Along the way, Beckus and Bellissard developed an abstract principle that helps derive such estimates \cite{bebe16}, which we shall aptly invoke and combine with Theorem \ref{thm:main-general} to prove Theorem~\ref{thm:gap-bds}.

\subsection{Gabor frames} Results in the style of
Theorems~\ref{thm:main-general} and \ref{thm:gap-bds} are typically
investigated in mathematical physics where $\delta $ represents  Planck's
constant or the strength of a magnetic field. While we hope that our results
may be useful in such questions, 
our main motivation comes from an open problem
in the theory of Gabor frames. Let $\Lambda \subseteq \rdd $ be a
discrete set, not necessarily a lattice, and consider the set of
phase-space shifts $\cG (g,\Lambda ) = \{\rho (\lambda )g\}_{ \lambda
\in 
\Lambda }$ for some $g\in \lrd $. The main problem is to understand  when $\cG (g,
\Lambda )$ is a frame, i.e., when the \emph{frame inequalities}
\begin{equation}
  \label{eq:c1}
A\|f\|_2^2 \leq \sum _{\lambda \in \Lambda } |\langle f, \rho (\lambda
)g \rangle |^2 \leq B \|f\|_2^2, \, \qquad \forall f\in \lrd ,
\end{equation}
hold for some constants $A,B>0$ independent of $f$.  
The optimal constants in \eqref{eq:c1} $B(\Lambda )$ and $A(\Lambda )$
are respectively the largest and the smallest spectral values of the frame operator
$S_{g,\Lambda } f = \sum_{\lambda\in \Lambda } \langle f, \rho
(\lambda )g \rangle \rho (\lambda )g$. Theorem~\ref{thm:main-general}
then leads to the following statement about the frame bounds of a
non-uniform Gabor frame (where ``non-uniform'' means that $\Lambda $
need not be a lattice). Again the most convenient conditions on $g$
are in terms of a \modsp .  
\begin{theorem} \label{blt}
Let $0<\alpha_0<1$,   $\alpha_0<\alpha<1/\alpha_0$, and $g\in M^1_{2}(\R^d)$. Let also $\Lambda\subset\R^{2d}$ be relatively separated, i.e., 
$
\emph{rel}(\Lambda):=\sup_{x\in\R^{d}}\#\{\lambda\in\Lambda\cap x+[0,1]^{d}\}<\infty$.
 Then 
 \begin{equation}\label{eq_ll}
 \left|\sigma_\pm(S_{g,\Lambda})-\sigma_\pm(S_{g,\alpha\Lambda})\right|\leq C_d\cdot \emph{rel}(\Lambda) \cdot\alpha_0^{-(4d+2)}\cdot\|g\|_{M^1_2}^2 \cdot|1-\alpha|.
 \end{equation}
\end{theorem}

Theorem~\ref{blt} has a rich history. If $\Lambda $ is a lattice, it was first shown in 
\cite{FK04} that the frame bounds depend in a \emph{lower semi-continuous} fashion on
$\alpha $, which implies that the set of lattices that generate a
Gabor frame is an open set.  For general
non-uniform sets $\Lambda$, the lower semi-continuity of the frame bounds was proven later in \cite{asfeka14}. 

A particularly important consequence of Theorem~\ref{blt} is the
quantitative behavior of the frame bounds near the critical
density. By the density theorem for Gabor frames, every frame $\cG
(g,\Lambda )$ must satisfy the necessary density condition
$D^-(\Lambda ) \geq 1$, where $D^-(\Lambda )$ is the lower Beurling
density and counts the average number of points per unit volume. If
$g\in M^1(\R^d)$, then even the strict inequality $D^-(\Lambda ) >1$
holds. In particular, if $g\in M^1(\R^d)$ and $D^-(\Lambda ) = 1$, then the
lower spectral bound of $S_{g,\Lambda }$ is $A(\Lambda ) = 0$. See~\cite{heil07} for a survey of the density theorem and \cite{asfeka14,grorcero15} for the relevant
result for non-uniform Gabor frames.
Since the ratio $B(\Lambda
)/A(\Lambda )$ serves as a condition number of the frame and thus
measures the stability of various reconstruction procedures, it is
important to understand how the lower frame bound $A(\Lambda )$
deteriorates to zero, as the density of $\Lambda $ decreases to
$1$. Theorem~\ref{blt} says that for an arbitrary set $\Lambda $ of
density $D^-(\Lambda ) = 1$ the lower frame bound of $\cG (g, \alpha
\Lambda )$ behaves like
$$
A(\alpha \Lambda ) \lesssim 1-\alpha,
$$
for $\alpha \to 1,\ \alpha<1$. This amounts to a blowup of the order $(1-\alpha)^{-1}$ for the condition number $B(\alpha\Lambda
)/A(\alpha\Lambda )$.

So far, this behavior of the lower frame bound near the critical density has
been proved with special methods only for Gabor frames for square 
lattices $\alpha \Z ^2$  based on the Gaussian $\varphi (t) = e^{-\pi t^2}$~\cite{BGL10},  and  the exponential functions $e^{-t}\chi _{[0,\infty )}$ and $e^{- |t|}$ \cite{KS14}. In hindsight, these results can be deduced from
Bellissard's result \cite{bel94}. Theorem~\ref{blt} fully settles the question: the
asymptotic behavior $
A(\alpha \Lambda ) \lesssim 1-\alpha
$ near the critical density holds for all Gabor frames $\cG (g,\alpha
\Lambda )$ with a window in $M^1_2(\R^d)$ and an arbitrary discrete set
$\Lambda $ of density $1$, be it a lattice or not, and in arbitrary dimensions. The condition on
the window is only slightly more restrictive than the general
condition $g\in M^1(\R^{d})$, under which the structural results on Gabor
frames hold.

Finally, we provide the following companion to Theorem \ref{blt}.
\begin{theorem}\label{thm4}
Under the assumptions of Theorem \ref{blt}, let $g$ be a gap of the spectrum of $S_{g,\Lambda}$ with length $L(g)$
and edges $\sigma^g_\pm(S_{g,{\Lambda}})$. Then there exist
$\varepsilon=\varepsilon(g)>0$ and gaps of the spectrum of 
$S_{g,{\alpha \Lambda}}$, $|1-\alpha|<\varepsilon$, whose edges
$\sigma^g_\pm(S_{g,{\alpha \Lambda}})$ satisfy
\begin{align}\label{eq_thm4}
\big|\sigma^g_\pm(S_{g,{\Lambda}})-
\sigma^g_\pm(S_{g,{\alpha \Lambda}})\big| \leq C_d \cdot |1-\alpha| \cdot
\mathrm{rel}(\Lambda)^2\cdot
 L(g)^{-1} \cdot \|g\|_{M^1_2}^4.
\end{align}
\end{theorem}
While gaps in the spectrum of Gabor frame operators are comparatively less studied than their spectral edges, general gaps deserve attention. For example, the size of the smallest gap $(0,A)$ of a Gabor frame operator coincides with the so-called \emph{lower Riesz bound} of the Gabor system. Note however that Theorem \ref{thm4} does not describe the size of the smallest gap, as it is in principle possible that a spectral gap $(0,A)$ of $S_{g,\Lambda}$ may evolve into a gap of $S_{g,\alpha\Lambda}$ with $\sigma^g_-(S_{g,{\alpha \Lambda}})>0$ leaving room for a second gap 
$(0,A')$ in the spectrum of $S_{g,{\alpha \Lambda}}$ with $A' < \sigma^g_-(S_{g,{\alpha \Lambda}})$.

\subsection{Methods}\label{sec_m}
As our predecessors, the overall structure of our proof of Theorem \ref{thm:main-general} follows Bellissard~\cite{bel94} and consists of three steps: a truncation argument followed by tensorization, and a reverse heat flow estimate. 

A key insight of Bellissard was that the $C^\ast$-algebra 
$\mathcal{A}_{1+\delta}$ that is generated by
$\{\rho(\sqrt{1+\delta}k)\}_{k\in\Z^2}$ acting on $\lrd $  (a non-commutative torus)   is isomorphic
to a subalgebra of $\mathcal{A}_{1}\otimes \mathcal{A}_\delta$. For
non-periodic symbols  we  use a similar tensorization argument. Although we
can no longer  rely on $C^\ast$-algebra techniques,  this difficulty
is circumvented  with the help of  the  metaplectic
representation. See Theorem~\ref{thm:tensor} and its comments.

While the overall structure of the proof of Theorem \ref{thm:main-general} is due to Bellissard, the proof techniques are rather different in the case of non-periodic symbols and draw from time-frequency analysis and the theory of modulation spaces. In fact, our  main technical contribution  is the systematic use of the machinery of \tfa. 

Theorem \ref{thm:gap-bds} is proved by combining Theorem \ref{thm:main-general} with an abstract principle due to Beckus and Bellissard \cite{bebe16}, while Theorems \ref{blt} and \ref{thm4} follow as a further application of the main results.

After understanding the fundamental principles underlying the
continuous dependence of spectra, Beckus and Bellissard with coauthors
have axiomatized and greatly expanded the range of their
methods~\cite{bebeni18, beta21}. In \cite{bebeni18} generalized
Schr\"odinger operators are studied in the context of groupoids and
dynamical systems based on $C^*$-algebraic methods. The final result
is a far-reaching general theorem about the continuity of the spectral
map that includes, for instance, Schr\"odinger operators for solids
with respect to general point distributions in magnetic
fields. The corresponding magnetic translations obey commutation
relations similar to those of the time-frequency shifts studied in our
work. In \cite{beta21} the authors study the Lipschitz continuity of
spectra of operator families for which the mapping from parameter to operator $\alpha \to
T_\alpha$  is driven by a dynamical system. This generalization is
motivated by and includes the almost-Mathieu operator, which in our
context corresponds to a specific finite sum of time-frequency
shifts. Our results are clearly related and fit into this general
context, but to the best of our knowledge there is no overlap.

\medskip

The paper is organized as follows: Section~\ref{sec:tools} summarizes the tools from
\tfa\ required for the formulation and for the  proof of the main
theorem. Section~\ref{sec:proofs} is devoted to the proof of the main theorems. 
Section~\ref{sec:gabor} discusses Gabor systems and their frame operators. 

\section{Background and Tools}\label{sec:tools}

\subsection{Notation}
Euclidean balls are denoted $B_r(x)$. The dilation operator acts on a
function $f: \mathbb{R}^d \to \mathbb{C}$ by $D_a f(x) = f(ax)$,
$a>0$, and for $F: \mathbb{R}^d \times \mathbb{R}^d \to \mathbb{C}$,
$\mathcal{U}$ denotes the change of variables
$\mathcal{U}F(x,\omega)=F(\omega,-x)$. The tensor product of two
functions is defined as $f\otimes g(s,t)=f(s)g(t)$. The symbol
$\lesssim $ in $f \lesssim g$ means that $f(x) \leq C g(x)$ for all
$x$ with a constant $C$ independent of $x$. 

\subsection{Norms and spectral extrema}
An important insight in \cite{bel94,bebe16} is that the smoothness of \emph{the norms} of (the polynomial calculi of) a family of operators determines the smoothness of their spectra in the Hausdorff metric. We start with the following basic estimate, which we prove for completeness; these and other closely related statements are implicit in the proofs of \cite[Theorem 2.3.17 and Theorem 2.8.10]{beck-thesis}.
\begin{lemma}\label{lem:norm-bd-edges}
	Let $\H$ be  a Hilbert space, and $A,A_1 ,A_2 \in B(\H)$ be  self-adjoint operators. For $\lambda>\|A\|_{B(\H)}$, we have
	$\sigma_+(A)=  \|A+\lambda I\|_{B(\H)}-\lambda$, and $\sigma_-(A)=\lambda-  \|A-\lambda I\|_{B(\H)}.$
	Moreover,  
	\begin{equation}\label{eq:diff-of-norms0}
		\left|\sigma_\pm(A_1)-\sigma_\pm(A_2)\right|\leq \|A_1-A_2\|_{ B(\H)},
	\end{equation}
	and 
	\begin{equation}\label{eq:diff-of-norms}
		\big|\|A_1\|_{ B(\H)}-\|A_2\|_{ B(\H)}\big|\leq 2\cdot\max_{\pm}|\sigma_\pm(A_1)-\sigma_\pm(A_2)|.
	\end{equation}
\end{lemma}
\proof
For a  self-adjoint operator $A$, it holds that
$r(A)=\|A\|_{B(\H)}$, where $r(A)$ denotes the spectral radius
of $A$.  For $\lambda>\|A\|_{B(\H)}$ one thus has
$\sigma_+(A)+\lambda=r(A+\lambda I)=\|A+\lambda
I\|_{B(\H)}$. For $\sigma_-$ we note that $-(\sigma_-(A)-\lambda)=r(A-\lambda I)=\|A-\lambda
I\|_{B(\H)}$.

Set  $\lambda= 2\max\{\|A_1\|_{B(\H)},\|A_2\|_{B(\H)}\}$. By the first part of this lemma and the triangle inequality it follows 
\begin{align*}
	\left|\sigma_\pm(A_1)-\sigma_\pm(A_2)\right|&=\left|\|A_1\pm\lambda I\|_{B(\H)}-\|A_2\pm\lambda I\|_{B(\H)}\right|
	\\
	&\leq \|A_1\pm\lambda I -(A_2\pm\lambda I)\|_{B(\H)}= \|A_1 -A_2\|_{B(\H)},
\end{align*}
which proves \eqref{eq:diff-of-norms0}. To prove \eqref{eq:diff-of-norms} we first note that
\begin{align*}
	\big|\|A_1\|_{B(\H)}\hspace{-1pt}-\hspace{-1pt}\|A_2\|_{B(\H)}\big|\hspace{-1pt} =\hspace{-1pt}\big|\max\{|\sigma_+(A_1))|,\hspace{-1pt}|\sigma_-(A_1)|\}\hspace{-1pt}-\hspace{-1pt}\max\{|\sigma_+(A_2)|,\hspace{-1pt}|\sigma_-(A_2)|\}\big|.
\end{align*}
For $a,b,c,d\geq 0$, the elementary estimate $|\max\{a,b\}-\max\{c,d\}|\leq 2\max\{|a-c|,|b-d|\}$ then yields that
\begin{align*}
	\big|\|A_1\|_{B(\H)}-\|A_2\|_{B(\H)}\big|&\leq 2  \max_\pm\big||\sigma_\pm(A_1)|-|\sigma_\pm(A_2)|\big|
	\\
	&\leq 2 \max_\pm|\sigma_\pm(A_1)-\sigma_\pm(A_2)|. 
\end{align*}
\pbox

\subsection{The Beckus-Bellissard lemma}
Consider a family $\{A_\delta\}_{|\delta|<  \delta_0 }$
of bounded self-adjoint operators $A_\delta: \H \to \H$ on a Hilbert space $\H$. For a set of polynomials $\mathcal{Q} \subset \mathbb{C}[x]$ let
\begin{align*}
C_\mathcal{Q}:=\sup_{p\in\mathcal{Q}}\sup_{\delta_1\neq\delta_2}\frac{\big|\|p(A_{\delta_1})\|_{B(\H)}-\|p(A_{\delta_2})\|_{B(\H)}\big|}{|\delta_1-\delta_2|}.
\end{align*}
We say that $\{A_\delta\}_{\delta \in (-\delta_0,\delta_0)}$ is
$(p2)$-\emph{Lipschitz continuous} if for every $M>0$, $C_{\mathcal{P}_M}<\infty$, where $\mathcal{P}_M$ denotes the set of all polynomials of the form $p(x)=\alpha x^2+\beta x+\gamma$ with $\alpha,\beta,\gamma\in\R$ and $|\alpha|+|\beta|+|\gamma|\leq M$.

The following lemma is a minor modification of \cite[Lemma~10]{bebe16} (see also \cite[Theorem~2.8.10]{beck-thesis}).
\begin{lemma}\label{lem:beckus-bellissard}
	Let $\{A_\delta\}_{|\delta|<\delta_0}$ be a $(p2)$-Lipschitz continuous family of bounded self-adjoint operators and set 
	\begin{align}\label{eq_aaaa}
	\mathcal{P}(A_0)=\left\{ p(x)=x^2+\beta x+\gamma:
	\beta,\gamma\in \mathbb{R},
	\ |\beta|\leq 2\|A_0\|_{B(\H)},\ |\gamma|\leq 5\|A_0\|^2_{B(\H)} \right\}.
	\end{align}
	Let $g$ be a gap of the spectrum of $A_0$ with edges $\sigma_\pm^g(A_0)$ and
	length $L(g)$. Then there exist $\varepsilon=\varepsilon(g)>0$ and gaps of the spectrum of $A_\delta$, $|\delta|<\varepsilon$, whose edges $\sigma_\pm^g(A_\delta)$ satisfy
	\begin{equation}
		\big| \sigma_\pm^g(A_\delta)-\sigma_\pm^g(A_0)\big|\leq 3|\delta|\cdot\frac{C_{\mathcal{P}(A_0)}}{L(g)}, \qquad |\delta|<\varepsilon.
	\end{equation}
\end{lemma}
The statement of Lemma~10 in \cite{bebe16} involves a larger class than $\mathcal{P}(A_0)$ defined by imposing the same bound on all polynomial coefficients. The proof in \cite{bebe16}, however, readily gives the stronger statement, and will not be repeated.

\subsection{Time-frequency representations}
We now offer a minimalist account of \tfa , modulation spaces, and
the associated results for pseudodifferential operators.
Detailed expositions can be found in the textbook~\cite{groe1} and the two
recent monographs~\cite{BO20,CRbook}.  

For a point $z=(x,\omega)\in \R^{2d}$,  the \emph{phase-space shift}
(\tfs ) of $f$ is defined as
$$
\rho (z)f(t)=e^{-i\pi x\cdot \omega } M_\omega T_xf(t)=e^{-i\pi x\cdot
  \omega } \, e^{2\pi i \omega\cdot t}f(t-x),
$$
where $T_xf(t)=f(t-x)$ and $M_\omega f(t)=e^{2\pi i \omega\cdot t}f(t)$.
In terms of the \emph{symplectic form}
\begin{align}\label{eq_symplectic}
[z,z'] = x'\cdot \omega - x\cdot \omega', \qquad z=(x,\omega), z'=(x',\omega') \in \mathbb{R}^d \times \mathbb{R}^d,
\end{align}
the composition of two phase-space shifts gives
\begin{align}\label{eq_x}
\rho(z)\rho(z')=e^{i\pi[z,z']} \rho(z+z'),
\qquad z,z'\in\mathbb{R}^{2d}.
\end{align}
In particular $\rho(z)^*=\rho(-z)$ and
\begin{align}\label{eq_x2}
\rho(z)^*\rho(z')\rho(z)=e^{2\pi i [z',z]} \rho(z'),
\qquad z,z'\in\mathbb{R}^{2d}.
\end{align}
The \emph{short-time Fourier transform} of a function or distribution
$f$ on $\R^d$  with respect to a window function  $g$ is given by
\begin{align*}
V_gf(x,\omega)&=\int_{\R^d}f(t)\overline{g(t-x)}e^{-2\pi i \omega\cdot t}dt\\
&=\langle f,M_\omega T_xg\rangle=e^{-i\pi x\cdot \omega } \, \langle f,\rho (z)g\rangle .
\end{align*}
When $g$ is normalized by $\norm{2}{g}=1$, $V_g:L^2(\mathbb{R}^d) \to L^2(\mathbb{R}^{2d})$ is an isometry ~\cite{fo89,groe1}: 
\begin{align}\label{eq_iso}
\int_{\rdd } |\langle f, \rho(z)g\rangle |^2 = \int_{\rdd } |V_g f(z)|^2 \, dz = \|f\|_2^2, \qquad f \in L^2(\mathbb{R}^d).
\end{align}
In terms of the rank-one projections $q(z) \in B(L^2(\R^d))$,
\begin{align}\label{eq_q}
q(z)=\langle\hspace{0.1cm} \cdot\hspace{0.05cm}, \rho(z)g \rangle  \rho(z) g,
\end{align}
the isometry property of the \stft\, yields the following continuous resolution of the identity:
\begin{equation}
\label{eq:c4}
\int _{\rdd} q(z) \, dz = I,
\end{equation}
where integrals are to be interpreted in the weak sense \eqref{eq_iso}.

\emph{Identities for the short-time Fourier transform.} We will  need  the following identity for the short-time Fourier transform of a pointwise product of functions
 \begin{equation}\label{eq:stft-prod}
 V_g(f\cdot h)(x,\omega)=\big(\widehat{h}\ast_2 V_gf\big)(x,\omega) = \int_{\R^d}\widehat{h}(\xi)V_gf(x,\omega-\xi)d\xi \, .
\end{equation}
The \emph{(cross-) Wigner distribution} of $f,g\in L^2(\R^d)$ is defined to be 
\begin{align}\label{eq_wigner}
\mathcal{W}(f,g)(x,\omega)=\int_{\R^d}f\Big(x+\frac{t}{2}\Big)\overline{g\Big(x-\frac{t}{2}\Big)}e^{-2\pi i \omega\cdot t}dt.
\end{align}
If $f=g$, we write $\mathcal{W}(f)$.

Finally we quote some useful facts about the \stft\ of a \stft . It
is shown in \cite[Lemmas 2.1 and 2.2, Proposition 2.5]{cogroe03}
that~\footnote{The combination $\mathcal{U} ^{-1}  \mathcal{F} $ is often called the symplectic Fourier
  transform and used in addition to $\mathcal{F} $.}
\begin{equation}\label{eq:fourier-prod-stft}
\mathcal{F}\big(V_\varphi f \cdot\overline{V_\varphi
  g}\big)(z)=\mathcal{U}^{-1}\big(V_gf\cdot\overline{V_\varphi\varphi}\big)(z),
\qquad z\in \rdd \, ,
\end{equation}
while with $x,\omega \in \rdd $,
$(\tilde{\omega}_1,\tilde{\omega}_2) = (\omega_2,-\omega_1)$, and the window
$\Phi=W(\varphi,\varphi)$, 
\begin{equation}
\label{eq_hhh}
|V_\Phi(\mathcal{W}(f,g))(x,\omega)|=|V_\varphi f(x-\tilde{\omega}/2)V_\varphi
g(x+\tilde{\omega}/2)|\,.
  \end{equation}

\subsection{Modulation spaces}  The family of \modsp s $M^{p,q}_{s,t}(\R^d)$
was already defined in \eqref{def:modul}. By changing the order of
integration, we obtain the family of \emph{Wiener amalgam spaces}.
Let $\varphi(t)=2^{d/4}e^{-\pi |t|^2}$ denote the standard Gaussian in
$\R^d$ and $ 1\leq p,q\leq\infty,\ s,t\geq
0$. Then     $W^{p,q}_{s,t}(\R^d)$ consists  of all distributions in $\mathcal{S}'(\R^d)$ for which
the following norm is finite
\begin{equation}\label{eq:wiener}
\|f\|_{W^{p,q}_{s,t}}=\left(\int_{\R^d}\left(\int_{\R^d}|V_\varphi f(x,\omega)|^p(1+|\omega|)^{sp}d\omega \right)^{q/p}(1+|x|)^{tq}dx\right)^{1/q},
\end{equation}
with the usual modification when $p=\infty $, or
 $q=\infty$.  If $p=q$, we write $W^p_{s,t}(\R^d)$,  if  $s=t$, we write $W^{p,q}_{s}(\R^d)$, and if $s=t=0$, we write $W^{p,q}(\R^d)$. 
Any nonzero function $g\in \mathcal{S}(\R^d)$ instead of the Gaussian in \eqref{eq:wiener} gives an equivalent norm, i.e. $\|f\|_{W^{p,q}_{s,t}}\asymp \|\mathcal{U}V_g f\|_{L^{p,q}_{s,t}}$; these spaces are often denoted by $W(\mathcal{F}L^p_s,L^q_t)$.

Since $V_gf(x,\omega ) = e^{-2\pi i x\cdot \omega } V_{\widehat{g}}\widehat{f}(\omega ,-x)$, a
comparison of \eqref{def:modul} and \eqref{eq:wiener} shows that   $W^{p,q}_{s,t}(\R^d)$ is the image of the  modulation space $M^{p,q}_{s,t}(\R^d)$ under the Fourier transform, in particular 
\begin{equation}
  \label{eq:c6}
\|f\|_{M^{p,q}_{s,t}} \asymp \big\|\widehat{f}\, \big\|_{W^{p,q}_{s,t}}.
\end{equation}

\emph{Convolution and multiplication in \modsp s.} 
The space $M^{\infty,1}$ is isometrically translation invariant. As a consequence, it satisfies $L^1 * M^{\infty,1} \to M^{\infty,1}$ together with the estimate
 \begin{equation}
	\label{eq:c7p}
	\| f \ast g\|_{M^{\infty,1}} \lesssim \|f \|_{L^{1} } \,
	\|g\|_{M^{\infty,1}}.
\end{equation}
Equivalently, in terms of Wiener amalgam norms,
\begin{equation}
	\label{eq:c8}
	\| f\cdot g\|_{W^{\infty,1}} \lesssim \|f\|_{\mathcal{F} L^1} \,
	\|g\|_{W^{\infty,1}},
\end{equation}
where \[\|f\|_{\mathcal{F} L^1} :=  \| \widehat{f}\,  \|_{L^1}.\]
We will also use the following estimates for convolution of functions and distributions in \modsp s, taken from \cite[Prop. 2.4]{cogroe03}.
  If $f\in M^\infty(\R^d) $ and $g\in M^1_{0,s}(\R^d)$ with $s\geq 0$, then $f \ast g \in M^{\infty,1}_{0,s}(\R^d)$ and
 \begin{equation}
 \label{eq:c7}
 \| f \ast g\|_{M^{\infty,1}_{0,s}} \lesssim \|f \|_{M^\infty } \,
 \|g\|_{M^{1}_{0,s}}\, .
 \end{equation}
Let us write $X_i$ to denote the multiplication operator $X_i f(t)=t_i f(t),\ 1\leq i\leq d$. The observation
$$
X_i (M_\omega T_xf)=M_\omega T_x (X_i f )+x_i M_\omega T_xf,
$$
and the fact that different windows generate equivalent norms for
$M^{p,q}_{s,t}(\R^d)$ and $W^{p,q}_{s,t}(\R^d)$ lead to the estimates
\begin{align}\label{eq:X-implies-weight}
 \| X_if \|_{M^{p,q}_{s,t}}\lesssim   \|f  \|_{M^{p,q}_{s+1,t}},\quad \text{and}\quad
 \| X_if \|_{W^{p,q}_{s,t}}\lesssim   \|f  \|_{W^{p,q}_{s,t+1}}.
\end{align}
Similarly, one sees that
$$
\| \partial _i f \|_{M^{p,q}_{s,t}}\lesssim   \|f
\|_{M^{p,q}_{s,t+1}} \, .
$$
Consequently,
\begin{equation}
  \label{eq:ju1}
  \|X_i \partial _if \|_{M^{p,q}_{s,t}}\lesssim   \|f
  \|_{M^{p,q}_{s+1,t+1}} \, .
\end{equation}
Next, we show the following variant of \cite[Proposition~2.5]{cogroe03}.
\begin{lemma}\label{lem:aux-W(G)}
For $f,g\in M^1_{s+t}(\rd )$,  $s,t\geq 0$, we have $\mathcal{W}(f,g)
\in M^1_{s,t}(\rdd )$ with the norm estimate 
$$
\|\mathcal{W}(f,g)\|_{M_{s,t}^1}\lesssim \|f\|_{M^1_{s+t}}\|g\|_{M_{s+t}^1}.
$$
\end{lemma}
\proof
Using \eqref{eq_hhh} with $\Phi=\mathcal{W}(\varphi,\varphi)$
and the change of variables $(\tilde{\omega}_1,\tilde{\omega}_2) = (\omega_2,-\omega_1)$,
\begin{align*}
  \|\mathcal{W}(f,g)\|_{M_{s,t}^1} &\asymp \int_{\rdd } \int _{\rdd }
                                     |V_\Phi
                                     (\mathcal{W}(f,g))(x,\omega )|
                                     (1+|x|)^s(1+|\omega|)^t dxd\omega \\
  &= \int_{\rdd}\int_{\rdd}|V_\varphi f(x-\omega/2)V_\varphi g(x+\omega/2)|(1+|x|)^s(1+|\omega|)^t dxd\omega
\\
&=\int_{\rdd}\int_{\rdd}|V_\varphi f(x)V_\varphi g(x+\omega)|(1+|x+\omega/2|)^s(1+|\omega|)^tdxd\omega
\\
&=\int_{\rdd}\int_{\rdd}|V_\varphi f(x)V_\varphi g(\omega)| \, (1+|(x+\omega)/2|)^s(1+|\omega-x|)^{t}d\omega dx
\\
&\lesssim \|  f\|_{M^1_{s+t}}\|g\|_{M^1_{s+t}},
\end{align*}
where we used the submultiplicativity of the polynomial weights.\pbox

As we note below, the space $M^\infty(\R^d)$ contains atomic measures supported on relatively separated sets.
\begin{lemma}\label{lemma_mu}
	Let $\Lambda \subset \mathbb{R}^{d}$ be relatively separated. Then $\mu := \sum_{\lambda \in \Lambda} \delta_\lambda \in M^{\infty}(\mathbb{R}^d)$. Moreover, if $\emph{rel}(\Lambda):=\sup_{x\in\R^{d}}\#\{\lambda\in\Lambda\cap x+[0,1]^{d}\}$ then
	$$
	\|\mu\|_{M^\infty}\lesssim \emph{rel}(\Lambda).
	$$
\end{lemma}
The proof of Lemma \ref{lemma_mu} follows from a direct calculation (which is easily carried out by taking a window function $g\in M^1(\R^d)$ supported on $[0,1]^d$) and is therefore omitted.

Finally, we will need the embeddings
\begin{equation}
 \label{eq:L-infty-M-infty} 
 \|f\|_{M^\infty}\lesssim \|f\|_{L^\infty}\lesssim \|f\|_{M^{\infty,1}},    
\end{equation} 
and the following norm estimates for dilations on modulation spaces:
\begin{align}\label{eq_c}
	&\|D_a f\|_{M^{\infty,1}_{0,s}} \leq C_{d,s}  \max\big\{1,a^{d+s}\big\} \|f\|_{M^{\infty,1}_{0,s}}, \qquad
	a>0,\ s\geq0 \, ,
	\\
	\label{eq_cx}
	&\|D_a f\|_{M^{1}_{0,s}} \leq C_{d,s}  \max\big\{a^{-d},a^{s}\big\} \|f\|_{M^{1}_{0,s}}, \qquad
	a>0,\ s\geq0 \, .
\end{align}
See \cite[Theorem~1.1]{suto07} and \cite[Theorem 3.2]{coou12} for the weighted versions.

\subsection{Pseudodifferential operators} The \modsp s $M^{\infty
  ,1}_{0,s}(\rdd ),\ s\geq 0$, are important symbol classes in the
theory of pseudodifferential operators. In particular, the space  $M^{\infty
  ,1}(\R^{2d})$ was first used by Sj\"ostrand~\cite{Sjo94} as a class of
non-smooth (``rough'') symbols that contains the H\"ormander class
$S^{0}_{0,0}$. See also the early papers \cite{Gro06,GR06} for a
detailed \tfa\  of this symbol class.

The standard definition of the Weyl calculus~\eqref{eq:c0} does not reveal
how \modsp s and phase-space methods enter the analysis.  This becomes
more plausible when we write a pseudodifferential operator as
$$
\langle \sigma ^w f , g\rangle   = \langle \sigma , \mathcal{W}(g,f)
\rangle
$$
or as a superposition of phase-space shifts
$$\sigma ^w = \int _{\rdd } \mathcal{U} \widehat{\sigma }(z) \rho (z) \, dz
$$
with $\mathcal{U} F(x,\omega ) = F(\omega ,-x)$.
Taking these formulas for the
Weyl calculus as the starting point, the appearance of \modsp s is
natural and ultimately led to the following results, which we will use
in an essential way. 

The composition of Weyl transforms defines a bilinear form on the space of symbols \emph{(twisted product)}
$$
\sigma^w\tau^w=(\sigma\, \sharp \, \tau)^w.
$$
Using  the symplectic form \eqref{eq_symplectic},
the twisted product of Schwartz class symbols can be written explicitly as
\begin{align}\label{eq_tp}
\sigma\, \sharp \, \tau(z) = 4^d \int_{\mathbb{R}^{2d}}\int_{\mathbb{R}^{2d}} \sigma(z') \tau(z'') e^{4 \pi i [z-z',z-z'']} \, dz' dz'',
\end{align}
while for general $\sigma$ and $\tau$ this formula holds in the distributional sense. The \emph{twisted convolution}
\begin{align*}
\sigma \,\natural\, \tau (z) = \int_{\mathbb{R}^{2d}} \sigma(z') \tau(z-z') e^{-\pi i [z-z',z']} \, dz'
\end{align*}
is related to the twisted product by
\begin{align}\label{eq_tpc}
\mathcal{F} \big( \sigma\, \sharp \, \tau \big) = \big(\mathcal{F} \sigma \big) \, \natural \, \big(\mathcal{F} \tau\big).
\end{align}
We now quote some basic properties of weighted Sj\"ostrand classes.
\begin{theorem}\label{lem:bdd}
(i) If $\sigma \in M^{\infty,1}(\rdd )$, then $\sigma $ is a bounded  operator on
$\lrd $ and
$$
\|\sigma ^w\|_{B(L^2(\R^d))} \lesssim \|\sigma
\|_{M^{\infty ,1}} \, .
$$

(ii) If $F\in W^{\infty ,1}(\R^{2d})$ and $\rho (F ) = \int _{\rdd } F(z) \rho
(z) \, dz $, then $\rho (F)$ is bounded on $\lrd $ with operator norm
$\|\rho (F) \|_{B(\lrd )} \lesssim \|F\|_{W^{\infty ,1}}$.

(iii) For $s \geq 0$, $M_{0,s}^{\infty,1}(\rdd)$ is a Banach $\ast$-algebra with respect to the twisted product $\sharp$ and the involution $\sigma\mapsto\overline{\sigma}$. In particular, $\|\sigma\, \sharp \, \tau\|_{M_{0,s}^{\infty,1}}\lesssim\|\sigma\|_{M_{0,s}^{\infty,1}}\|\tau\|_{M_{0,s}^{\infty,1}}$.

(iv)   Let $|\delta|<\delta_0<1$, $G_\delta\in
M^{\infty,1}(\R^{2d})$  be real-valued, and set $T_\delta =
(D_{\sqrt{1+\delta }}G_\delta )^w$. 
Then 
\begin{align}\label{eq_a}
\|T_\delta\|_{B(L^2(\R^d))}\lesssim  \max\{1,( 1+\delta )^{ {d} }\}
\|G_\delta\|_{M^{\infty,1}}\leq ( 1+\delta_0)^{ {d} }
\| {G_\delta}\|_{M^{\infty,1}}. 
\end{align}
\end{theorem}
\proof For (i) and (iii), see \cite{Gro06,Sjo94,GR06}. Item (ii) is just a
reformulation when the operator is written as a superposition of
phase-space shifts.  (iv) follows from the
invariance of $M^{\infty ,1}(\R^{2d})$ under dilations expressed by
\eqref{eq_c}. Note that symbols are functions on $\rdd $ and the
correct norm of the dilation is therefore $\sqrt{1+\delta } ^{2d} $. \pbox

\section{Proof of the Main Results}\label{sec:proofs}

In our proof of Theorems~\ref{thm:main-general}  we  follow Bellissard's
strategy ~\cite{bel94} consisting of three main ingredients: (i) truncate the symbol of $T_\delta$ to define an operator $\mathcal{T}_R(T_\delta)$ and    compare its spectral extreme values  to the ones of  $T_\delta$ (Lemma~\ref{lem:trunc-error}), (ii)  introduce a tensorization $ \mathcal{T}_R(T_\delta)^\otimes$ acting on $L^2(\R^{2d})$ which preserves the spectrum of $\mathcal{T}_R(T_\delta)$ (Theorem~\ref{thm:tensor}), and (iii) rewrite $ \mathcal{T}_R(T_\delta)^\otimes$ adequately (Lemma~\ref{lemma_new}) so that reverse
heat-flow estimates can be used to compare the spectral extreme values of $ \mathcal{T}_R(T_\delta)^\otimes$ and $\mathcal{T}_R(T_0)$ (Lemmas~\ref{lem:Sd-Td_times} and \ref{lem:Sd-T0}). 

The challenge in our case is to find suitable
alternative arguments to treat non-periodic symbols. As a first step,
we write the pseudodifferential operator $T_\delta $ with symbol
$\sigma _\delta = D_{\sqrt{1+\delta }} G_\delta $  as a
superposition of phase-space shifts (spreading representation in engineering
language)
$$
T_\delta = \int _{\rdd } \mathcal{U} \widehat{\sigma _\delta }(z) \rho
(z) \, dz =  \int_{\R^{2d}}\mathcal{U} \widehat{G_\delta}(z)
\rho\Big(\sqrt{1+\delta}\,z\Big)\,dz \, .
$$
Although in general $\widehat{G_\delta }$  is a distribution, the
analysis of $T_\delta$ becomes feasible in this representation. A main
tool is the boundedness estimate from  Theorem~\ref{lem:bdd}, which we
will use several times. We will at first assume that $\delta$ is positive and use reflection arguments to cover negative values.
 
\subsection{Truncation error}
Fix a real-valued, even, radial function $\theta\in
C^\infty(\R^{2d})$ such that $\theta(z)=1$, for $|z|\leq 1$,
$\theta(z)=0$, for $|z|\geq 2$,  and $0\leq \theta(z)\leq 1$ else, and
set $\theta_R(z) =\theta(z/R)$, with $R>0$.

We define the truncation of $T_\delta$ by
\begin{align}\label{eq_trunc}
\mathcal{T}_R(T_{\delta})=\int_{\R^{2d}}\mathcal{U}(\theta_R \widehat{G_\delta})(z) \rho\left(\sqrt{1+\delta}\,z\right)\,dz.
\end{align}
Since $\theta $ is real-valued and even, it follows that
$\mathcal{T}_R(T_{\delta})$ is self-adjoint. By Theorem~\ref{lem:bdd}, it
is enough to bound $\big\|G_\delta-\widehat{\theta_R}\ast
G_\delta\big\|_{M^{\infty,1}}$ in order to derive an estimate of the
norm of $T_\delta-\mathcal{T}_R(T_{\delta})$.

\begin{lemma}\label{lem:trunc-error}
Let $0\leq \delta  < \delta_0$ and   $G_\delta\in M^{\infty,1}_{0,2}(\R^{2d})$ be  real-valued. Assume that $\sup_{|t|<\delta_0}\| G_t \|_{M^{\infty,1}_{0,2}}<\infty$, then
$$
\big\|G_\delta-\widehat{\theta_R}\ast G_\delta\big\|_{M^{\infty,1}}\lesssim  R^{-2}\cdot\sup_{|t|<\delta_0}\| G_t \|_{M^{\infty,1}_{ 0,2}}.
$$
In particular, by \eqref{eq_a},
\begin{align}\label{eq_b}
\|T_\delta-\mathcal{T}_R(T_{\delta})\|_{B(L^2(\R^d))}\lesssim   R^{-2}\cdot(1+\delta_0)^d\cdot\sup_{|t|<\delta_0}\| G_t \|_{M^{\infty,1}_{0,2}}.
\end{align}
\end{lemma}
\proof Fix a constant $K>0$ and choose $\Phi\in M^1(\R^{2d})$ to be compactly supported
in $B_K(0)$. Let $R \geq 2K$. If $|x|\leq R-K$, then $\mathrm{supp}\, T_x \Phi
\subseteq B_R(0)$ and consequently $V_\Phi
\widehat{G_\delta}(x,\omega)-V_\Phi(\theta_R
\widehat{G_\delta})(x,\omega)=0$ for $|x| \leq R-K$ and all $\omega
\in \R^{2d} $. Therefore,
\begin{align*}
\big\|G_\delta-\widehat{\theta_R}&\ast G_\delta\big\|_{M^{\infty,1}}  =\hspace{-1pt}
\big\|\widehat{G_\delta}-\theta_R  \widehat{G_\delta}\big\|_{W^{\infty,1}} \\
& \asymp  \int_{\R^{2d}}\sup_{\omega\in\R^{2d}} \left|V_\Phi  \widehat{G_\delta}(x,\omega)-V_\Phi(\theta_R  \widehat{ G_\delta})(x,\omega)\right|dx
\\
&=\hspace{-1pt}\int_{\R^{2d}\backslash B_{R-K}(0)}\sup_{\omega\in\R^{2d}} \left|V_\Phi  \widehat{G_\delta}(x,\omega)-V_\Phi (\theta_R  \widehat{ G_\delta})(x,\omega)\right|dx
\\
&\lesssim  R^{-2}  \int_{\R^{2d}} \sup_{\omega\in\R^{2d}}  \Big( |V_\Phi \widehat{G_\delta}(x,\omega)| + |\widehat{\theta_R}\ast_2  V_\Phi  \widehat{ G_\delta} (x,\omega) | \Big)(1 + |x|)^2dx
\\
&\leq  R^{-2}(1+\|\widehat{\theta_R}\|_1) \int_{\R^{2d}}\sup_{\omega\in\R^{2d}}  |V_\Phi \widehat{G_\delta}(x,\omega)|  (1+|x|)^2dx
\\
&\asymp  R^{-2}(1+\|\widehat{\theta_R}\|_1) \|\widehat{G_\delta}\|_{W^{\infty,1}_{0,2}}
\\
&\asymp
     R^{-2}(1+\|\widehat{\theta}\|_1)\|G_\delta\|_{M^{\infty,1}_{0,2}}
     \, .
\end{align*}
We have used that $R \leq 2(R-K)$, \eqref{eq:stft-prod}, and Young's inequality to show that $\sup _\omega
|\widehat{\theta_R}\ast_2 V_\Phi  \widehat{ G_\delta} (x,\omega)| \leq
\|\theta _R\|_1 \sup _\omega | V_\Phi\widehat{ G_\delta} (x,\omega)| $, and
$$
\big\| \widehat{\theta_R}\big\|_1=R^{2d}\int_{\R^{2d}}\big|\widehat{\theta}(R\omega)\big|d\omega=\big\|\widehat{\theta}\big\|_1.
 $$
Finally, for $0 \leq R \leq 2K$,
\begin{align*}
\big\|G_\delta-\widehat{\theta_R}\ast G_\delta\big\|_{M^{\infty,1}} & =\hspace{-1pt}
\big\|\widehat{G_\delta}-\theta_R  \widehat{G_\delta}\big\|_{W^{\infty,1}} \\
& \asymp \hspace{-1pt}\int_{\R^{2d}}\sup_{\omega\in\R^{2d}} \left|V_\Phi  \widehat{G_\delta}(x,\omega)-V_\Phi(\theta_R  \widehat{ G_\delta})(x,\omega)\right|dx
\\
&\lesssim \int_{\R^{2d}}\sup_{\omega\in\R^{2d}} \Big( |V_\Phi \widehat{G_\delta}(x,\omega)|+ |\widehat{\theta_R}\ast_2V_\Phi  \widehat{ G_\delta} (x,\omega) | \Big) dx
\\
&\leq\hspace{-1pt} (1+\|\widehat{\theta_R}\|_1) \int_{\R^{2d}}\sup_{\omega\in\R^{2d}}  |V_\Phi \widehat{G_\delta}(x,\omega)|  dx
\\
&\asymp
(1+\|\widehat{\theta}\|_1)\|G_\delta\|_{M^{\infty,1}}
\\
&\leq \frac{4K^2}{R^2}(1+\|\widehat{\theta}\|_1)\|G_\delta\|_{M^{\infty,1}_{0,2}},
\end{align*}
which concludes the proof after adjusting the implied constants. 
\pbox

The following estimate will be helpful to analyze truncations.
\begin{lemma}\label{lemma_x}
	For $0 < R \leq \delta^{-1/2}$, we have $\big\|e^{\frac{\pi\delta}{2} |\cdot|^2} \theta_R\|_{\mathcal{F}L^1} \lesssim 1$.
\end{lemma}
\begin{proof}
	We use the dilation invariance of $\mathcal{F}L^1$ and the Sobolev-type embedding
	$M^2_{0,d+1}(\mathbb{R}^{2d})$ $\hookrightarrow \mathcal{F}L^1(\mathbb{R}^{2d})$. Using multi index notation for derivatives, we obtain
	\begin{align*}
		&\big\|e^{\frac{\pi\delta}{2} |\cdot|^2} \theta(\cdot/R)\|_{\mathcal{F}L^1} = 
		\big\|e^{\frac{\pi\delta R^2}{2} |\cdot|^2} \theta\|_{\mathcal{F}L^1}
		\\
		&\qquad\lesssim\sum_{|\alpha|,|\beta| \leq d+1}
		\big\|
		\partial^\alpha \big[ e^{\frac{\pi\delta R^2}{2} |\cdot|^2} \big]
		\cdot
		\partial^\beta \theta \big\|_{L^2}
		\lesssim\sum_{|\alpha| \leq d+1}
		\big\|
		\partial^\alpha \big[ e^{\frac{\pi\delta R^2}{2} |\cdot|^2} \big] \big\|_{L^\infty(B_2(0))} \lesssim 1,
	\end{align*}
	because $\delta R^2 \leq 1$.
\end{proof}

\subsection{Tensorization}

Let $\delta>0$ and define $T_{  \delta}^\otimes:L^2(\R^{2d})\rightarrow L^2(\R^{2d})$
\begin{align*}
T_{  \delta}^\otimes&=\int_{\R^{2d}} \mathcal{U}\widehat{G_\delta} (z)\rho( z)\otimes \rho(\sqrt{\delta}z)dz
\\
&=\int_{\R^{2d}} \mathcal{U}\widehat{G_\delta}(x,\omega)\rho( x,\sqrt{\delta}x,\omega,\sqrt{\delta}\omega)dxd\omega.
\end{align*}
Note that if $G_\delta$ is real-valued, then both $T_{ \delta}$ and $T_{ \delta}^\otimes$ are
self-adjoint. We emphasize that the tensorized operator $T_\delta
^\otimes$ acts on $L^2(\rdd )$, whereas the original operator
$T_\delta $ acts on $\lrd $. We similarly define a tensorized operator associated with the truncation \eqref{eq_trunc}:
\begin{align*}
\mathcal{T}_R(T_{\delta})^\otimes=\int_{\R^{2d}}\mathcal{U}(\theta_R\widehat{G_\delta})(z)
\rho (z) \otimes \rho(\sqrt{\delta}z)dz.
\end{align*}

Let $ \text{Sp}(d)$ denote the \emph{symplectic group} of all matrices $\mathcal{R}\in \text{GL}(2d,\R)$ that satisfy $\mathcal{R}^\ast J\mathcal{R}=J$, with $J=\left(\begin{smallmatrix}0 & I_d\\ -I_d &0\end{smallmatrix}\right)$. For each symplectic matrix, there exists a unitary operator $\mu(\mathcal{R})$, called \emph{metaplectic operator}, such that 
\begin{equation}\label{eq:meta-symp}
\rho(\mathcal{R}(x,\omega))= \mu(\mathcal{R})\rho(x,\omega)\mu(\mathcal{R})^{-1},
\end{equation}
see for example \cite[Lemma~9.4.3]{groe1}.
We are now ready to prove that $T_{\delta}$ and $T_{\delta}^\otimes$ are isospectral.

\begin{theorem}\label{thm:tensor}
Let $0\leq \delta <\delta_0$, and  $G_\delta \in  M^{\infty,1}(\R^{2d})$ be real-valued. Then 
$\sigma(T_{  \delta})=\sigma(T_\delta^\otimes)$, and
$\sigma(\mathcal{T}_R(T_{\delta}))=\sigma(\mathcal{T}_R(T_{\delta})^\otimes)$, for all $R>0$.
\end{theorem}
\proof
Let ${R}_\delta\in \mathcal{O}(2d)$ be an  orthogonal transformation
that satisfies 
$$
R_\delta(
x, \sqrt{\delta}x
) = (
\sqrt{1+\delta}x,0
),\quad \text{for every }x\in \R^{d}, 
$$ 
and set $\mathcal{R}_\delta=\left(\begin{smallmatrix}
R_\delta & 0\\ 0 & R_\delta
\end{smallmatrix}\right)\in\R^{4d\times 4d}.
$
Then $\mathcal{R}_\delta\in\text{Sp}(2d)$ and
\begin{align*}
 \mu(\mathcal{R}_\delta)
 \rho(x,\sqrt{\delta}x,\omega,\sqrt{\delta}\omega)\mu(\mathcal{R}_\delta)^{-1}
 &= \rho \big(R_\delta (x,\sqrt{\delta}x), R_\delta
 (\omega,\sqrt{\delta}\omega )\big) \\ &  =
 \rho(\sqrt{1+\delta}x,0,\sqrt{1+\delta}\omega,0) \, .    
\end{align*}
 Since $\rho(\sqrt{1+\delta}x,0,\sqrt{1+\delta}\omega,0)=\rho(\sqrt{1+\delta}x,\sqrt{1+\delta}\omega)\otimes I$,  it follows that 
\begin{align*}
T_\delta\otimes I& =\int_{\R^{2d}}\mathcal{U}\widehat{G_\delta}(z)\rho(\sqrt{1+\delta}x,0,\sqrt{1+\delta}\omega,0)dz
\\
&= \int_{\R^{2d}}\mathcal{U}\widehat{G_\delta}(z)   \mu(\mathcal{R}_\delta) \rho(x,\sqrt{\delta}x,\omega,\sqrt{\delta}\omega)\mu(\mathcal{R}_\delta)^{-1}dz
\\
&=\mu(\mathcal{R}_\delta) T_\delta^\otimes  \mu(\mathcal{R}_\delta)^{-1}.
\end{align*}
Hence  $T_\delta\otimes I$ and $T_\delta^\otimes$ are unitarily equivalent and therefore $\sigma(T_\delta)=\sigma(T_\delta\otimes I)=\sigma(T_\delta^\otimes)$. The same argument applies to $\mathcal{T}_R(T_{\delta})$.
\pbox

Bellissard \cite{bel94} proved a special case of Theorem~\ref{thm:tensor} for periodic symbols with $C^*$-algebra arguments. Our
main insight is that the metaplectic representation allows one to
treat also non-periodic symbols (and perhaps provides a more direct
argument even for periodic ones). 

We now extend the resolution of the identity \eqref{eq:c4} to obtain
the following expansion of phase-space shifts. Recall that  $\varphi$
is always  the normalized Gaussian
\begin{align}\label{eq_gauss}
\varphi(t)=2^{d/4}e^{-\pi |t|^2}, \quad t \in \mathbb{R}^d.
\end{align}
The following proposition, which can be found in \cite[Proposition~2~(vi)]{bel94}, is the core of the argument leading to Theorem \ref{thm:main-general}. We provide a short proof for the reader's convenience.

\begin{lemma}\label{lemma_q}
	Let $\varphi$ be the normalized Gaussian \eqref{eq_gauss}, $[\cdot,\cdot]$ the symplectic form
	\eqref{eq_symplectic}, and $q$ the rank-one projection \eqref{eq_q} associated with $\varphi$. Then
	\begin{align}\label{eq:tf-multiplier}
	\rho(z) =e^{\frac{\pi}{2} |z|^2} \int_{\R^{2d}} e^{2\pi i[z,z']} q(z') dz', \qquad z \in \mathbb{R}^{2d},
	\end{align}
where integral converges in the weak sense.
\end{lemma}
\proof
Let $f,g\in L^2(\R^d)$.  Using \eqref{eq:fourier-prod-stft}, we find that
\begin{align*}
  \int_{\R^{2d}} e^{2\pi i[z,z']} \langle q(z')f,g\rangle  dz' &=    \mathcal{F}  
 \Big( V_\varphi f \, \overline{V_\varphi g} \Big) (-\omega ,x) \\  
&=  \mathcal{U}   \mathcal{F}   \Big( V_\varphi f \, \overline{V_\varphi g} \Big) (-x,-\omega) \\  
&= V_gf(-z)\overline{V_\varphi \varphi (-z)} = \langle \rho (z) f, g \rangle \, e^{-\pi  |z|^2/2}  \, .
\end{align*}
Since $f$ and $g$ were arbitrary, this implies \eqref{eq:tf-multiplier}. 
\pbox

Next, we apply Lemma \ref{lemma_q} to inspect the tensorized operator $T_{\delta}^\otimes$.

\begin{lemma}\label{lemma_new}
Fix $G_\delta \in  M^{\infty,1}(\R^{2d})$ for  $\delta> 0$ and 
 set
\begin{equation}
\label{eq:c3}
Q_R( T_{ \delta})(z'):=\int_{\R^{2d}}e^{\frac{\pi\delta}{2}
	|z|^2}e^{2\pi i[z,z']} \mathcal{U}(\theta_R\widehat{G_\delta})(z)
\rho(z) dz.
\end{equation}
Then
\begin{align}\label{eq:tensor-T-Q}
\mathcal{T}_R(T_{\delta})^\otimes=\frac{1}{\delta^d}\int_{\R^{2d}}{Q}_R (T_{ \delta})(z')\otimes q(
z'/\sqrt{\delta})dz'.
\end{align}
\end{lemma}
\begin{proof}
	Let us first observe that if $h(z)=z\cdot z'$, then  $\mathcal{U}h(z)=[z,z']$ which implies that for $f,g\in L^2(\R^d)$
\begin{equation}\label{eq:F-Q}
\langle Q_R(T_\delta)(z')f,g\rangle =\mathcal{U}\mathcal{F}\big(e^{\frac{\pi\delta}{2}
	|\cdot|^2} \mathcal{U}(\theta_R\widehat{G_\delta})\langle \rho(\cdot) f,g\rangle\big)(z'),
\end{equation}
as well as
$$
\int_{\R^{2d}}e^{2\pi i \sqrt{\delta}[z,z']}\langle q(z')f,g\rangle dz'=D_{\sqrt{\delta}}\mathcal{U}\mathcal{F}^{-1}(\langle q(\cdot)f,g\rangle)(z).
$$
Using \eqref{eq:tf-multiplier} for the second factor in $\rho
(z) \otimes \rho (\sqrt{\delta}z)$,  we may formally write $\mathcal{T}_R(T_{ \delta})^\otimes$ as
\begin{align*}
\mathcal{T}_R(T_{
  \delta})^\otimes&=\int_{\R^{2d}}\mathcal{U}(\theta_R\widehat{G_\delta})(z)
                    \rho (z) \otimes \rho(\sqrt{\delta}z)dz \notag \\
 &=\int_{\R^{2d}}\mathcal{U}(\theta_R\widehat{G_\delta})(z) \rho (z)
   \otimes \left(e^{\frac{\pi\delta}{2} |z|^2} \int_{\R^{2d}} e^{2\pi
   i \sqrt{\delta}  [z,z']} q(z') dz'\right)dz \notag 
\\
 &=\int_{\R^{2d}}e^{\frac{\pi\delta}{2} |z|^2} \mathcal{U}(\theta_R\widehat{G_\delta})(z) \rho (z)
   \otimes  \int_{\R^{2d}} e^{2\pi
   i \sqrt{\delta}  [z,z']} q(z') dz' dz. \notag 
\end{align*}
For $f_1,f_2,g_1,g_2\in L^2(\R^d)$, we therefore get
\begin{align*}
\big\langle \mathcal{T}_R(T_{ \delta})^\otimes& (f_1\otimes f_2), (g_1\otimes g_2)\big\rangle 
\\
&=\Big\langle \mathcal{U}(e^{\frac{\pi\delta}{2} |\cdot|^2}\theta_R\widehat{G_\delta})\langle \rho(\cdot)f_1,g_1\rangle, \overline{D_{\sqrt{\delta}}\mathcal{U}\mathcal{F}^{-1}\big(\langle q(\cdot)f_2,g_2\rangle \big)
}\Big\rangle
\\
&=\frac{1}{\delta^d}\Big\langle  \mathcal{F}^{-1} \mathcal{U}^\ast\big( \mathcal{U}(e^{\frac{\pi\delta}{2} |\cdot|^2}\theta_R\widehat{G_\delta})\langle \rho(\cdot)f_1,g_1\rangle\big), \overline{D_{1/\sqrt{\delta}}\big(\langle q(\cdot)f_2,g_2\rangle} \big)
\Big\rangle
\\
&=\frac{1}{\delta^d}\Big\langle \mathcal{U} \mathcal{F} \big( \mathcal{U}(e^{\frac{\pi\delta}{2} |\cdot|^2}\theta_R\widehat{G_\delta})\langle \rho(\cdot)f_1,g_1\rangle\big), \overline{D_{1/\sqrt{\delta}}\big(\langle q(\cdot)f_2,g_2\rangle} \big)
\Big\rangle
\\
&=\frac{1}{\delta^d}\Big\langle \langle Q_R( T_{ \delta})(\cdot )  f_1,g_1\rangle, \overline{\langle q( \cdot /\sqrt{\delta})f_2,g_2\rangle}\Big\rangle
\\
&=\big\langle\frac{1}{\delta^d}\int_{\R^{2d}}{Q}_R (T_{ \delta})(z')\otimes q(
z'/\sqrt{\delta})dz'(f_1\otimes f_2), (g_1\otimes g_2)\big\rangle.
\end{align*}
To justify these calculations, we show that $\langle Q_R( T_{ \delta})(\cdot )  f_1,g_1\rangle\in L^\infty(\R^{2d})$ and that $\langle q( \cdot /\sqrt{\delta})f_2,g_2\rangle \in L^1(\R^{2d})$ (with norms that may depend on $\delta$ and $R$). 
To this end note first
\begin{align*}
\|\langle q(\cdot) f,g\rangle\|_{L^1} =  \big\|V_\varphi f\cdot\overline{V_\varphi g}\big\|_{L^1} 
\leq \big\|V_\varphi f\big\|_{L^2} \big\| {V_\varphi g}\big\|_{L^2} 
= \|f\|_2 \|g\|_2 \, .
\end{align*}
Second, letting for simplicity $\|f\|_2=\|g\|_2=1$, 
\begin{align*}
|\langle Q_R( T_{ \delta})(z )  f,g\rangle|&\leq \| Q_R( T_{ \delta})(z )\|_{B(L^2(\R^{d}))} 
\lesssim \big\|\mathcal{U}\big( M_{-z}e^{\frac{\pi\delta}{2} |\cdot|^2} \theta_R\widehat{G_\delta}\big)\big\|_{W^{\infty,1}} 
\\
& \lesssim  \big\|e^{\frac{\pi\delta}{2} |\cdot|^2} \theta_R\widehat{G_\delta}\big\|_{W^{\infty,1}} \lesssim    \big\|e^{\frac{\pi\delta}{2} |\cdot|^2} \theta_R\|_{\mathcal{F}L^1}\|\widehat{G_\delta}\big\|_{W^{\infty,1}}  
<\infty,
\end{align*}
where we  used \eqref{eq:c8} and $e^{\frac{\pi\delta}{2} |\cdot|^2} \theta_R \in \mathcal{S}(\mathbb{R}^{2d})$. Taking the supremum over $z\in\R^{2d}$ shows that $\langle Q_R( T_{ \delta})(\cdot )  f,g\rangle\in L^\infty(\R^{2d})$.
\end{proof}

We need the following estimate on the reversal of the heat-flow, which follows from a Taylor expansion; see, e.g., \cite[Lemma 6]{bel94}.

\begin{lemma}\label{lem:conv-aux}
Let $F\in C_b^2(\R^{2d})$ and   $\Phi_\delta(z)=\frac{1}{\delta^d}e^{-{\pi}|z|^2/\delta}$ with $z \in \mathbb{R}^{2d}$ and $\delta>0$. Then
\begin{equation}
\|F-\Phi_\delta\ast F\|_\infty\lesssim \delta\|\partial^2F\|_\infty:=\delta\sum_{|\alpha|=2}\|\partial^\alpha F\|_\infty.
\end{equation}
\end{lemma}

Finally, we compare the spectral extreme values of the tensorization $\mathcal{T}_R(T_\delta)^\otimes$ and $Q_R(T_{ \delta})(0)$.

\begin{lemma}\label{lem:Sd-Td_times}
Let $0< \delta   < \delta_0$
and  $G_\delta \in  M^{\infty,1}_{0,2}(\R^{2d})$ be real-valued.  Assume that $\sup_{|t|<\delta_0}\|G_t\|_{M^{\infty,1}_{0,2}}< \infty $ and $0<R \leq \delta^{-1/2}$, then
$$
\big|\sigma_{\pm}\big(Q_R(T_{ \delta})(0)\big)-\sigma_\pm\big(\mathcal{T}_R(T_{\delta})^\otimes\big)\big|\lesssim  \delta  \cdot\sup_{|t|<\delta_0}\left\|   G_t  \right\|_{M^{\infty,1}_{0,2}}.
$$
\end{lemma}
\proof \textbf{Step 1.}
Note that $Q_R(T_{ 0} )(0)=\mathcal{T}_R(T_{ 0})$. In addition, by \eqref{eq_x2},
\begin{equation}\label{eq:Sz-S0}
Q_R( T_{ \delta})(z')=\rho(z')^\ast Q_R (T_{ \delta}) (0)\rho(z').
\end{equation}
Consequently, $\sigma\big(Q_R(T_{ \delta}) (z)\big)=\sigma\big(Q_R(T_{
  \delta})(0)\big)$, as  $Q_R(T_{ \delta})(z)$ and $Q_R(T_{ \delta})(0)$ are unitarily equivalent.

\smallskip

\noindent \textbf{Step 2.} We note that $Q_R(T_{ \delta})(z)$ is self-adjoint for every $z\in\R^{2d}$ because $Q_R(T_{ \delta})(0)$ is. Let $f\in L^2(\R^d)$ with $\|f\|_{2}=1$ and fix $\lambda\in\R$. Then, \eqref{eq:Sz-S0} shows that
\begin{align*}
\langle (Q_R(T_{ \delta})(z)-\lambda I) f,f\rangle &=\langle (Q_R(T_{ \delta})(0)-\lambda I)\rho(z)f,\rho(z)f\rangle \\
&\leq \|Q_R(T_{ \delta})(0)-\lambda I\|_{B(L^2(\R^{d}))}.
\end{align*}
Hence,
$(Q_R(T_{ \delta})(z)-\lambda I)\otimes q(z/\sqrt{\delta})\leq
\|Q_R(T_{ \delta})(0)-\lambda I\|_{B(L^2(\R^{d}))} I\otimes
q(z/\sqrt{\delta})$. Next, we 
invoke Lemmas \ref{lemma_q} 
and  \eqref{eq:tensor-T-Q} of Lemma~\ref{lemma_new} to obtain 
\begin{align}
\mathcal{T}_R(T_{\delta})^\otimes-\lambda I\otimes I &=\frac{1}{\delta^d}\int_{\R^{2d}}Q_R(T_{ \delta})(z)\otimes q(z/\sqrt{\delta})dz-\lambda I\otimes I
\notag
\\
&=\frac{1}{\delta^d}\int_{\R^{2d}}(Q_R(T_{ \delta})(z)-\lambda I)\otimes q(z/\sqrt{\delta})dz \notag
\\
&\leq\|Q_R(T_{ \delta})(0)-\lambda I\|_{B(L^2(\R^{d}))}\frac{1}{\delta^d}\int_{\R^{2d}}I\otimes q(z/\sqrt{\delta})dz \notag
\\
&=\|Q_R(T_{ \delta})(0)-\lambda I\|_{B(L^2(\R^{d}))} I\otimes I . \label{eq:T-I}
\end{align}
Repeating this argument
for $\lambda I\otimes I-\mathcal{T}_R(T_{\delta})^\otimes$ shows that 
\begin{equation}\label{eq:T_leq_S}
\|\mathcal{T}_R(T_{\delta})^\otimes-\lambda I\otimes I \|_{B(L^2(\R^{2d}))}\leq \|Q_R(T_{ \delta})(0)-\lambda I\|_{B(L^2(\R^{d}))}.
\end{equation}

\smallskip

\noindent \textbf{Step 3.} For a lower estimate let $f\in \lrd$, $\|f\|_2=1$, and
apply Lemma~\ref{lemma_new}:
\begin{align*}
\|\mathcal{T}_R (T_{\delta})^\otimes &-  \lambda I\otimes I\|_{B(L^2(\R^{2d}))}
\\
 &\geq \left|\Big\langle (\mathcal{T}_R(T_{\delta})^\otimes-\lambda
   I\otimes I)(f\otimes \varphi),(f\otimes \varphi)\Big\rangle\right| 
\\
&=\left|\frac{1}{\delta^d}\int_{\R^{2d}}\langle Q_R(T_\delta)(z)f,f\rangle \, \langle q(z/\sqrt{\delta}) \varphi ,
     \varphi \rangle   dz\, - \lambda \|f\|_2^2 \|\varphi \|_2^2 \right|   = (*)      \, .
\end{align*}
Since the \stft\ of a Gaussian is again a Gaussian, we find that  $\langle q(z/\sqrt{\delta}) \varphi ,
     \varphi \rangle  = 
|\langle\varphi,\rho(z/\sqrt{\delta})\varphi\rangle|^2 =  e^{-\pi
  |z|^2/\delta}=\delta^d\Phi_\delta(z)$, and that the last
expression involves a convolution with the scaled Gaussian $\Phi _\delta$.
We can continue as follows: 
\begin{align*}
 (*)  &= \left|\int_{\R^{2d}}\langle Q_R(T_{ \delta})(z)f,f\rangle
        \Phi_\delta( 0-z )dz \, - \lambda \right|
\\
&=\left|\big(\langle  Q_R(T_{ \delta})(\cdot) f,f\rangle \ast \Phi_\delta\big)(0)-\lambda\right|
\\
&\geq  \left|\langle  Q_R(T_{ \delta}) (0) f,f\rangle -\hspace{-0.05cm}\lambda\right|-\left\|\langle  Q_R(T_{ \delta}) (\cdot) f,f\rangle- \langle  Q_R(T_{ \delta})(\cdot) f,f\rangle\ast \Phi_\delta \right\|_\infty
\\
&\geq  \left|\langle  Q_R(T_{ \delta}) (0) f,f\rangle - \lambda\right| -  \sup_{\|h\|_2=1} \left\| \langle  Q_R(T_{ \delta}) (\cdot) h,h\rangle- \langle  Q_R(T_{ \delta})(\cdot) h,h\rangle\ast \Phi_\delta \right\|_\infty.
\end{align*}
In the first term we take the supremum over $f\in L^2(\R^d)$,
$\|f\|_2=1$; to the second term we apply
Lemma~\ref{lem:conv-aux}. In view of \eqref{eq:T_leq_S} this  leads to
\begin{align*}
\Big|\|\mathcal{T}_R(T_{\delta})^\otimes -\lambda I\otimes I\|_{B(L^2(\R^{2d}))}&- \|Q_R(T_{ \delta})(0)-\lambda I\|_{B(L^2(\R^{d}))}\Big|
\\
&\leq    \sup_{\|h\|_2=1}\left\|  \langle  Q_R(T_{ \delta})(\cdot) h,h\rangle- \langle  Q_R(T_{ \delta})(\cdot) h,h\rangle\ast \Phi_\delta  \right\|_\infty
\\
&\lesssim \delta   \sup_{\|h\|_2=1}\|\partial^2  \langle  Q_R(T_{ \delta}) (\cdot) h,h\rangle  \|_\infty.
\end{align*}
The deductions above together with the first part of Lemma~\ref{lem:norm-bd-edges}
(for an appropriate choice of $\lambda$) then show
\begin{align}\label{eq_d}
\big|\sigma_{\pm}\big(Q_R(T_{ \delta})(0)\big)-\sigma_\pm\big(\mathcal{T}_R(T_{\delta} )^\otimes\big)\big|\lesssim \delta   \sup_{\|h\|_2=1}\|\partial^2  \langle  Q_R(T_{ \delta}) (\cdot) h,h\rangle  \|_\infty.
\end{align}
\textbf{Step 4.}
It remains to further estimate the right-hand
side of \eqref{eq_d}. Let $\|h\|_2=1$.
Using \eqref{eq:F-Q}, the partial derivatives of $\langle Q_R(T_\delta )(z') h,h\rangle$ are given  as follows:
with the notation $z'_i=(x'_i, \omega'_i) \in \mathbb{R}^2$ and $i=1,\ldots,d$,
let $u'_i$ be either $x'_i$ or $\omega'_i$. Then
\begin{align*}
\big|\partial_{u_i'}\partial_{u_j'}&\big\langle Q_R(T_\delta))(z')h,h\big\rangle \big|=\big|\partial_{u_i'}\partial_{u_j'}\mathcal{U}\mathcal{F}\big(e^{\frac{\pi\delta}{2}|\cdot|^2}\mathcal{U}(\theta_R \widehat{G_\delta})\langle\rho(\cdot)h,h\rangle\big)(z')\big|
\\
&=4\pi^2 \big|\mathcal{U}\mathcal{F}\big(X_{i'}X_{j'}\,e^{\frac{\pi\delta}{2}|\cdot|^2}\mathcal{U}(\theta_R \widehat{G_\delta})\langle\rho(\cdot)h,h\rangle\big)(z')\big|,
\end{align*}
for suitable indices $i',j' \in \{1,\ldots,2d\}$.
We apply
Theorem~\ref{lem:bdd}(ii) and obtain
\begin{align*}
 \|\partial^2   \langle  Q_R(T_{ \delta})& (\cdot) h,h\rangle  \|_\infty 
  \\
 &\lesssim\sup_{z'\in\R^{2d}}  \sum_{i,j=1,...,2d}\left| \int_{\R^{2d}}\mathcal{U}\big( M_{-z'}e^{\frac{\pi\delta}{2} |\cdot|^2}X_iX_j \theta_R\widehat{G_\delta}\big)(z) \langle\rho(z)h,h\rangle dz\right|
  \\
 &\lesssim   \sup_{z'\in\R^{2d}}  \sum_{i,j=1,...,2d}\left\|\mathcal{U}\big( M_{-z'} e^{\frac{\pi\delta}{2} |\cdot|^2}    X_i X_j \theta_R  \widehat{G_\delta} \big) \right\|_{W^{\infty,1}}
   \\
 &\lesssim    \sum_{i,j=1,...,2d}\left\|  e^{\frac{\pi\delta}{2}
   |\cdot|^2}    X_i X_j \theta_R  \widehat{G_\delta}
   \right\|_{W^{\infty,1}} \, .
\end{align*}
For  each of the terms we  use the product
property~\eqref{eq:c8} and obtain
$$
\|\partial^2   \langle  Q_R(T_{ \delta}) (\cdot) h,h\rangle  \|_\infty 
   \lesssim      \sum_{i,j=1,...,2d} \Big\|  e^{\frac{\pi\delta}{2}
     |\cdot|^2}    \theta_R \Big\|_{\mathcal{F}L^1}\Big\| X_i
   X_j\widehat{G_\delta}\Big\|_{W^{\infty,1}}\, .
   $$
By Lemma \ref{lemma_x}, $\|  e^{\frac{\pi\delta}{2}
	|\cdot|^2}    \theta_R \|_{\mathcal{F}L^1} \lesssim 1$ for $R\leq \delta^{-1/2}$, whereas, by \eqref{eq:X-implies-weight},
\begin{align}\label{eq:XX-weight}
\big\| X_i X_j\widehat{G_\delta}\big\|_{W^{\infty,1}}\lesssim  \big\| \widehat{G_\delta} \big\|_{W^{\infty,1}_{0,2}}
\asymp \|G_\delta\|_{M^{\infty,1}_{0,2}}.
\end{align}
In conclusion, we have shown that 
\begin{align*}
 \big\|\partial^2  &\langle  Q_R(T_{ \delta})(\cdot) h,h\rangle  \big\|_\infty
 \lesssim \|G_\delta\|_{M^{\infty,1}_{0,2}},
 \end{align*}
which, combined with \eqref{eq_d}, completes the proof.
\pbox

\subsection{Differentiation of the symbol}
\begin{lemma}\label{lem:Sd-T0}
Assume that $G_\delta\in M^{\infty,1}_{0,2}(\R^{2d})$ is real-valued, $\delta\mapsto G_\delta$ is differentiable, $\partial_\delta  G_\delta\in M^{\infty,1}(\R^{2d})$ for $0<\delta < \delta _0 <1$, and $0<R\leq\delta^{-1/2}$. Then
\begin{equation}\label{eq:Sd-T0}
\big\| Q_R(T_{ \delta})(0)- \mathcal{T}_R(T_{0}) \big\|_{B(L^2(\R^d))}\lesssim  \delta\cdot\sup_{|t|<\delta_0}\left(\left\|  G_t\right\|_{M^{\infty,1}_{0,2}} +\left\| \partial_t G_t\right\|_{ M^{\infty,1}} \right).
\end{equation}
\end{lemma}
\proof
Recall that
$$
Q_R(T_{ \delta})(0)-\mathcal{T}_R(T_{0}) = \int _{\rdd } e^{\pi \delta
  |z|^2/2} \cU (\theta _R \widehat{G_\delta })(z) \, \rho (z) \, dz - 
\int _{\rdd } \cU (\theta _R \widehat{G_0 })(z) \, \rho (z) \, dz
\, .
$$
Using  Theorem~\ref{lem:bdd}~(ii),  we estimate the operator norm by 
$$
\big\| Q_R(T_{ \delta})(0)-\mathcal{T}_R(T_{0})\big\|_{B(L^2(\R^d))} \lesssim 
 \Big\|e^{\frac{\pi\delta}{2} |\cdot|^2} \theta_R  \widehat{G_\delta}-
 \theta_R  \widehat{G_0}\Big\|_{W^{\infty,1}} \, .
 $$
Although $\widehat{G_\delta }$ is a tempered distribution, the \stft\  $H(\delta,x,\omega) \\ =V_\varphi \big(e^{\frac{\pi\delta}{2} |\cdot|^2} \theta_R  \widehat{G_\delta}\big)(x,\omega)$, $x,\omega\in\R^{2d}$,
is a smooth function and therefore we may express it as  
$$
H(\delta,x,\omega)-H(0,x,\omega)=\int_0^\delta  \partial_tH(t,x,\omega)dt.
$$ 
Then by the definition of the $W^{\infty ,1}$-norm we have 
 \begin{align*}
  \Big\|e^{\frac{\pi\delta}{2} |\cdot|^2} \theta_R  \widehat{G_\delta}-  \theta_R  \widehat{G_0}\Big\|_{W^{\infty,1}}
 &= \int_{\R^{2d}}\sup_{\omega\in \R^{2d}}\left|H(\delta,x,\omega)-H(0,x,\omega) \right|dx
 \\
 &\leq  \int_0^\delta \int_{\R^{2d}} \sup_{\omega\in\R^{2d}} \left| \partial_t H(t,x,\omega)\right|dxdt
 \\
 &\leq \delta \sup_{|t|\leq \delta} \int_{\R^{2d}} \sup_{\omega\in\R^{2d}} \left| \partial_tH(t,x,\omega)\right|dx.
\end{align*}
Spelling out $\partial_tH(t,\cdot )$ explicitly and using the identity
$V_{\widehat{g}}\widehat{f}(x,\omega)=e^{-2\pi i
  x\cdot\omega}V_gf(-\omega,x)$  with $\widehat{f} = e^{\pi \delta |\cdot |^2/2}
\theta _R \widehat{G_\delta}$ and $\widehat{g} = \widehat{\varphi }=\varphi $
gives
\begin{align}
 \partial_t H(t,x,\omega )&=e^{-2\pi i x\cdot\omega}\partial_tV_\varphi\big(\mathcal{F}^{-1}\big( e^{\frac{\pi t}{2}|\cdot|^2}\theta_R\big)\ast G_t\big)(-\omega,x)\notag
 \\
 &=e^{-2\pi i x\cdot \omega}  V_\varphi\Big( \mathcal{F}^{-1}\big( \partial_te^{\frac{\pi t}{2}|\cdot|^2}\theta_R\big) \ast\hspace{-0.02cm} G_t\Big)(-\omega,x) \notag
 \\
 &\hspace{1.66cm}+e^{-2\pi i x\cdot \omega}  V_\varphi\Big(\mathcal{F}^{-1}\big( e^{\frac{\pi t}{2}|\cdot|^2}\theta_R\big) \ast  \partial_tG_t\Big)(-\omega,x)\notag
\\
& = \frac{\pi }{2}V_\varphi\left(|\cdot|^2  e^{ \frac{\pi t}{2}|\cdot|^2 }\theta_R\widehat{G_t}\right)(x,\omega )+V_\varphi\left(   e^{ \frac{\pi t}{2}|\cdot|^2 } \theta_R\widehat{\partial_t G_t}\right)(x,\omega).\label{eq:h-explicit}
\end{align}
In the calculations above, we interchanged integration (hidden in
$V_\varphi $)  and differentiation twice to obtain the second
equality. To justify this, we construct two integrable majorants and
apply the Leibniz integral rule. We  note that,  by regularity and support
assumptions on $\theta$,   
$$
\max\left\{\big|\mathcal{F}^{-1}\big(  e^{\frac{\pi t}{2}|\cdot|^2}\theta_R\big)(z)\big|,\big|\mathcal{F}^{-1}\big( |\cdot|^2 e^{\frac{\pi t}{2}|\cdot|^2}\theta_R\big)(z)\big|\right\}\leq C_R (1+|z|)^{-(2d+1)},
$$ for a constant $C_R$ independent of $t$, as long as $|t|\leq 1$, but dependent on $R$.  Therefore, using the embedding $M^{\infty
  ,1}(\R^{2d}) \subseteq L^\infty(\R^{2d}) $ from~\eqref{eq:L-infty-M-infty}, 
\begin{align*}
\Big|\mathcal{F}^{-1}\big( |\cdot|^2 e^{\frac{\pi t}{2}|\cdot|^2}&\theta_R\big)(z) G_t(w-z) +\mathcal{F}^{-1}\big(  e^{\frac{\pi t}{2}|\cdot|^2}\theta_R\big)(z)\partial_t G_t(w-z)  \Big|
\\ &\lesssim C_R (1+|z|)^{-(2d+1)}
\sup_{|t|<\delta_0}\big(\|G_t\|_\infty+ \|  \partial_t G_t\|_\infty \big)
\\ &\lesssim C_R (1+|z|)^{-(2d+1)}
\sup_{|t|<\delta_0}\big(\|G_t\|_{M^{\infty,1}}+ \| \partial_t G_t\|_{M^{\infty,1}} \big),
\end{align*}
as well as 
\begin{align*}
\Big|\Big(\mathcal{F}^{-1}&\big( |\cdot|^2 e^{\frac{\pi t}{2}|\cdot|^2}\theta_R\big)\ast G_t(y) +\mathcal{F}^{-1}\big(  e^{\frac{\pi t}{2}|\cdot|^2}\theta_R\big)\ast\partial_t G_t(y) \Big)\rho(z)\varphi(y) \Big|
\\
&\lesssim \Big|\Big(\big\|\mathcal{F}^{-1}\big( |\cdot|^2 e^{\frac{\pi t}{2}|\cdot|^2}\theta_R\big)\ast G_t\big\|_\infty +\big\|\mathcal{F}^{-1}\big(  e^{\frac{\pi t}{2}|\cdot|^2}\theta_R\big)\ast\partial_t G_t\big\|_\infty \Big) T_x\varphi(y) \Big|
\\
& 
 \lesssim C_R \sup_{|t|<\delta_0}\big(\|G_t\|_{M^{\infty,1}}+\|\partial _t G_t\|_{M^{\infty,1}}\big) |T_x\varphi(y)|.
\end{align*}
This provides the required majorants.

Finally, applying \eqref{eq:c8} to the expression \eqref{eq:h-explicit} and using Lemma \ref{lemma_x}
and \eqref{eq:X-implies-weight}, we obtain
\begin{align*}
 \| Q_R(T_{ \delta})(0)-&\mathcal{T}_R(T_{0})\|_{B(L^2(\R^d))}
 \\
 &\lesssim   \delta \sup_{|t|\leq \delta}\left( \big\|    e^{
   \frac{\pi t}{2} |\cdot|^2 }  \theta_R \, |\cdot|^2\, \widehat{G_t} \big\|_{W^{\infty,1}} 
+ \big\|e^{ \frac{\pi t}{2}  |\cdot|^2 }\,\theta_R\, \widehat{\partial_t  G_t}\big\|_{W^{\infty,1}} \right)
\\
 &\lesssim   \delta  \sup_{|t|<\delta_0}\left(\big\|  |\cdot|^2 \, {\widehat{G_t}} \big\|_{W^{\infty,1}} 
+ \big\| \widehat{\partial_t G_t}\big\|_{W^{\infty,1}} \right)
\\
 &=  \delta  \sup_{|t|<\delta_0}\left(\left\|  {G_t} \right\|_{M^{\infty,1}_{0,2}} 
+  \left\|\partial_t G_t\right\|_{M^{\infty,1}} \right).
\end{align*}
\pbox
 
\subsection{Proof of Theorem~\ref{thm:main-general}}
We only consider the right spectral extreme value; the proof for left   spectral extreme value works exactly in the same way. 

Let us first assume that $\delta=\delta_1 > 0$ and $\delta_2=0$, and let $R=\delta^{-1/2}$. We estimate
\begin{align*}
\big|\sigma_+&(T_\delta)-  \sigma_+(T_0)\big|
  \leq  \big|\sigma_+ (T_\delta)-\sigma_+(\mathcal{T}_R(T_{\delta}))\big|+
\big|\sigma_+(\mathcal{T}_R(T_{\delta}))-\sigma_+\big(\mathcal{T}_R(T_{\delta})^\otimes\big)\big|
 \\
 & +\big|\sigma_+\big(\mathcal{T}_R(T_{\delta})^\otimes\big)-\sigma_+\big(Q_R(T_{ \delta})(0)\big)\big| +\big|\sigma_+\big(Q_R(T_{ \delta})(0)\big)-\sigma_+(\mathcal{T}_R(T_{0}))\big| 
 \\
 &+\big|\sigma_+(\mathcal{T}_R(T_{0}))-\sigma_+(T_0)\big|.
\end{align*}
The first and last terms can be bounded using
Lemma~\ref{lem:trunc-error}. The second term is zero by Theorem~\ref{thm:tensor}. Lemma~\ref{lem:Sd-Td_times}  bounds the
third term, and Lemma~\ref{lem:Sd-T0} bounds the  fourth term. 
Altogether, we arrive at
\begin{align}\label{eq:end-delta-pos}
\big|\sigma_+(T_\delta)-\sigma_+(T_0)\big| & \lesssim   \left(\frac{1}{R^2}+ \delta\right)(1+\delta_0)^d\sup_{|t|<\delta_0}\left(\left\|  G_t\right\|_{M^{\infty,1}_{0,2}} +\left\| \partial_t G_t\right\|_{ M^{\infty,1}} \right)\notag
\\
& =2  \delta(1+\delta_0)^d\sup_{|t|<\delta_0}\left(\left\|  G_t\right\|_{M^{\infty,1}_{0,2}} +\left\| \partial_t G_t\right\|_{ M^{\infty,1}} \right).
\end{align}
Hence Lipschitz continuity of the spectral extreme values at $0$ holds for $\delta>0$.

The general case $-\delta_0<\delta_1\leq\delta_2<\delta_0$, $\delta_0<1$, needs an additional argument for which we
introduce a new parameter $\theta $. Fix $\delta_1,\delta_2$ and  define $\widetilde{T}_\theta$, $0\leq\theta< \theta_0=\frac{\delta_0-\delta_1}{1+\delta_1}$, via its Weyl symbol $D_{\sqrt{1+\theta}}\widetilde{G}_\theta$, where
$\widetilde{G}_\theta=D_{\sqrt{1+\delta_1}}G_{(1+\delta_1)\theta+\delta_1}$. For this choice, we have $\widetilde{T}_0=T_{\delta_1}$, and 
$\widetilde{T}_{(\delta_2-\delta_1)/(1+\delta_1)}=T_{\delta_2}$. 
By the dilation property~\eqref{eq_c} and $0<1+\delta_1<2$ we get
\begin{align*}
\big\|\widetilde{G}_\theta\big\|_{M^{\infty,1}_{0,2}}&=\big\|D_{\sqrt{1+ \delta_1}} G_{(1+\delta_1)\theta+\delta_1}\big\|_{M^{\infty,1}_{0,2}}\lesssim   \big\|  G_{(1+\delta_1)\theta+\delta_1}\big\|_{M^{\infty,1}_{0,2}}
\leq   \sup_{|t|<\delta_0} \left\|  G_t\right\|_{M^{\infty,1}_{0,2}},
\end{align*}
as well as
\begin{align*}
\big\|\partial_\theta\widetilde{G}_\theta\big\|_{M^{\infty,1}}
&= \big\|D_{\sqrt{1+\delta_1}} \partial_{\theta}[G_{(1+\delta_1)\theta+\delta_1}] \big\|_{M^{\infty,1}} \lesssim  \sup_{|t|<\delta_0}\|\partial_{t}G_{t}\|_{M^{\infty,1}}.
\end{align*}
Applying the estimate from \eqref{eq:end-delta-pos} to $\widetilde{T}_\theta$ and choosing in particular $\theta=\frac{\delta_2-\delta_1}{1+\delta_1} $ shows
\begin{align*}
\big|\sigma_+(T_{\delta_1})-  \sigma_+(T_{\delta_2})\big|&=\big|\sigma_+\big(\widetilde{T}_0\big)-\sigma_+\big(\widetilde{T}_{(\delta_2-\delta_1)/(1+\delta_1)}\big)\big| \notag
\\
&\lesssim  |\theta | (1+\theta_0)^d  \sup_{|t|<\delta_0}\left(\left\|
     G_t\right\|_{M^{\infty,1}_{0,2}} +\left\| \partial_t
     G_t\right\|_{ M^{\infty,1}} \right) \, . 
\end{align*}
The Lipschitz dependence follows from
$$
|\theta | (1+\theta _0)^d = \frac{|\delta_2-\delta_1|}{1+\delta_1 } \,
\Big(1+\frac{\delta _0-\delta_1}{1+\delta_1}\Big)^d \leq |\delta_1-\delta_2 |
\frac{2^d}{(1-\delta _0)^{d+1}} \, ,
$$
since $\delta _0<1$ and $1+\delta_1 > 1-\delta_0$. All in all,
$$
\big|\sigma_+(T_{\delta_1})-  \sigma_+(T_{\delta_2})\big|\lesssim
{ |\delta_1-\delta_2 | }{(1-\delta _0)^{-(d+1)}}\sup_{|t|<\delta_0}\left(\left\|
     G_t\right\|_{M^{\infty,1}_{0,2}} +\left\| \partial_t
     G_t\right\|_{ M^{\infty,1}} \right),
$$
as claimed.
\pbox

\subsection{Proof of Theorem~\ref{thm:gap-bds}}
Since the desired conclusion should hold for sufficiently small $\delta$, we may assume that $\delta_0<1/2$.
We shall apply the Beckus-Bellissard lemma (Lemma~\ref{lem:beckus-bellissard}). Consider a polynomial $p(x)=x^2+\beta x+\gamma$, with $\beta,\gamma\in\R$. The Weyl symbol of  $p({T}_\delta)$ is given by $D_{\sqrt{1+\delta}}\widetilde{G}_\delta$, where
$$
\widetilde{G}_\delta=D_{1/\sqrt{1+\delta}}\left( (D_{\sqrt{1+\delta}}G_\delta)\, \sharp \,  (D_{\sqrt{1+\delta}}G_\delta)\right)+\beta\cdot G_\delta+\gamma.
$$
Since $M^{\infty,1}_{0,2}(\R^{2d})$ is a  Banach algebra with unit element $1$ (Theorem~\ref{lem:bdd}~(iii)), it follows from \eqref{eq_c} and $\delta_0<1/2$ that
\begin{align}\label{eq_xaaa1}
\begin{aligned}
\big\|\widetilde{G}_\delta\big\|_{M^{\infty,1}_{0,2}}&\lesssim  \left(\|G_\delta\|_{M^{\infty,1}_{0,2}}^2+|\beta|\, \|G_\delta\|_{M^{\infty,1}_{0,2}}+|\gamma|\right).
\end{aligned}
\end{align}
Let us use the notation $G \, \sharp _\delta \,   H:=D_{ {1}/{\sqrt{1+\delta}}}\left( (D_{\sqrt{1+\delta}}G)\, \sharp \,  (D_{\sqrt{1+\delta}}H)\right)$. If $G$ and $H$ are Schwartz functions, a computation with \eqref{eq_tpc} gives
\begin{align*}
\mathcal{F} \big( G \, \sharp _\delta \,   H \big) (z)
= \int_{\mathbb{R}^{2d}} \widehat{G}(z') \widehat{H}(z-z') e^{-\pi (1+\delta)i [z-z',z']}\, dz'.
\end{align*}
Hence, assuming for a moment that $G_\delta$ is a Schwartz function,
\begin{align*}
&\mathcal{F} \big(\partial_\delta \big( G_\delta \, \sharp _\delta \,   G_\delta \big)\big) (z)
=
\partial_\delta \big(
\mathcal{F} \big( G_\delta \, \sharp _\delta \,   G_\delta \big) \big)(z)
\\
&\qquad=
\int_{\mathbb{R}^{2d}} \widehat{\partial_\delta G_\delta}(z') \widehat{G_\delta}(z-z') e^{-\pi (1+\delta)i [z-z',z']}\, dz'\, 
\\
&\qquad\qquad+\int_{\mathbb{R}^{2d}} \widehat{G_\delta}(z') \widehat{\partial_\delta G_\delta}(z-z') e^{-\pi (1+\delta)i [z-z',z']}\, dz'\, 
\\
&\qquad\qquad +\pi i \sum_{j=1}^d \int_{\mathbb{R}^{2d}} z'_{j+d} \cdot \widehat{G_\delta}(z') \cdot (z-z')_j \cdot \widehat{G_\delta}(z-z') e^{-\pi (1+\delta)i [z-z',z']}\, dz'\,
\\
&\qquad\qquad -\pi i \sum_{j=1}^d \int_{\mathbb{R}^{2d}} z'_{j} \cdot \widehat{G_\delta}(z') \cdot (z-z')_{j+d} \cdot \widehat{G_\delta}(z-z') e^{-\pi (1+\delta)i [z-z',z']}\, dz'.
\end{align*}
Consequently,
\begin{align}\label{eq_hp}
\partial_\delta \big( G_\delta \, \sharp _\delta \,   G_\delta \big) =
\big( (\partial_\delta G_\delta) \, \sharp _\delta \,   G_\delta \big)
+
\big( G_\delta \, \sharp _\delta \, (\partial_\delta G_\delta )\big)
+ \sum_{k=1}^{2d} \lambda_k (\partial_{z_{j_k}} G_\delta) \, \sharp _\delta \, (\partial_{z_{j'_k}} G_\delta),
\end{align}
for suitable indices $j'_k,j_k \in \{0,\ldots,2d\}$ and $\lambda_k \in \mathbb{C}$, with $|\lambda_k|=1/2$. To see that  \eqref{eq_hp} is valid in general, we take a sequence $G_\delta^k\in \mathcal{S}(\R^{2d})$ such that $\partial_\delta G_\delta^k\in \mathcal{S}(\R^{2d})$,   $G_\delta^k\stackrel{w^\ast}{\rightarrow}G_\delta$ in $M^{\infty,1}_{0,2}(\R^{2d})$, and   $\partial_\delta G_\delta^k\stackrel{w^\ast}{\rightarrow}\partial_\delta G_\delta$ in $M^{\infty,1}(\R^{2d})$. For example, define $G_\delta^k(z):=\psi(z/k)\cdot (G_\delta\ast \phi_{k})(z)$, where $\psi\in\mathcal{S}(\R^{2d})$ is chosen such that $0\leq \psi\leq 1$, and $\psi(z)=1$ for $z\in B_1(0)$, and $\phi_{k}(z)=k^{2d}\phi(k z )$ for a mollifier $\phi$. 

Applying consecutively  \eqref{eq_c}, \eqref{eq:X-implies-weight}, and Theorem~\ref{lem:bdd}~(iii), it thus follows
\begin{align*}
	\|\partial_\delta (G_\delta  \, &\sharp _\delta \,   G_\delta)\|_{M^{\infty,1}}
	\\
	& \lesssim  
	 \|D_{\sqrt{1+\delta}} G_\delta  \,\sharp\, D_{\sqrt{1+\delta}} (\partial_\delta  G_\delta) \|_{M^{\infty,1}}\,	\\
	&\qquad +\max_{j,j'=1,\ldots,2d}\big\|(D_{\sqrt{1+\delta}}(\partial_{z_j} G_\delta))\,\sharp\,(  D_{\sqrt{1+\delta}}(\partial_{z_{j'}} G_\delta))\big\|_{M^{\infty,1}}
	\\
	& \lesssim  
	 \|G_\delta\|_{M^{\infty,1}}\left\| \partial_\delta G_\delta\right\|_{M^{\infty,1}} +\|G_\delta\|_{M^{\infty,1}_{0,1}}^2
	\\
	& \lesssim
	 \|G_\delta\|_{M^{\infty,1}}\left\| \partial_\delta G_\delta\right\|_{M^{\infty,1}}+\|G_\delta\|_{M^{\infty,1}_{0,2}}^2.
\end{align*}
Therefore,
\begin{align}\label{eq_bbb1}
		\big\|  {\partial_\delta}\widetilde{G}_\delta\big\|_{M^{\infty,1}}&\leq \|\partial_\delta (G_\delta  \, \sharp _\delta \,   G_\delta)\|_{M^{\infty,1}}+|\beta|\,\|\partial_\delta G_\delta \|_{M^{\infty,1}}\notag
		\\
		&\lesssim
		 \|G_\delta\|_{M^{\infty,1}}\left\| \partial_\delta G_\delta\right\|_{M^{\infty,1}}+\|G_\delta\|_{M^{\infty,1}_{0, 2}}^2
		+ |\beta| \, \|\partial_\delta G_\delta \|_{M^{\infty,1}}.
\end{align}
Let $\delta_1,\delta_2 \in (-\delta_0,\delta_0 )\subset (-1/2,1/2)$.
Combining Theorem~\ref{thm:main-general} with \eqref{eq_xaaa1} and \eqref{eq_bbb1},
\begin{align*}
 \big|\sigma_\pm(p(T_{\delta_1})) &- \sigma_\pm(p(T_{\delta_2}))\big|
\lesssim  |\delta_1-\delta_2| 
\sup_{|\delta|<\delta_0} \left( 
\| \widetilde{G}_\delta \|_{M^{\infty,1}_{0,2}} + 
\|\partial_\delta \widetilde{G}_\delta \|_{M^{\infty,1}}\right)
\\
& \lesssim   |\delta_1-\delta_2| 
\sup_{|\delta|<\delta_0}
\Big[\left(\|G_\delta\|^2_{M^{\infty,1}_{0,2}}+\|G_\delta\|_{M^{\infty,1}}\|\partial_\delta G_\delta\|_{M^{\infty,1}}\right)
\\
&\qquad\qquad\qquad\qquad\qquad\qquad +
|\beta|\left(\|G_\delta\|_{M^{\infty,1}_{0,2}}+\|\partial_\delta G_\delta\|_{M^{\infty,1}}\right)+|\gamma|\Big].
\end{align*}
In particular, if $|\beta|\leq 2\|T_0\|_{B(L^2(\R^d))}$, $|\gamma|\leq 5\|T_0\|^2_{B(L^2(\R^d))}$, then,
by Theorem \ref{lem:bdd}, 
$|\beta| \lesssim \|G_\delta\|_{M^{\infty,1}}$ and
$|\gamma| \lesssim \|G_\delta\|_{M^{\infty,1}}^2 \leq \|G_\delta\|_{M^{\infty,1}_{0,2}}^2$.
Hence,
\begin{equation*}
{\big|\sigma_\pm(p(T_{\delta_1}))\hspace{-1pt} - \hspace{-1pt}\sigma_\pm(p(T_{\delta_2}))\big|}\lesssim \hspace{-1pt} |\delta_1\hspace{-1pt}-\hspace{-1pt}\delta_2| \sup_{|\delta|<\delta_0}\hspace{-2pt}\left(\|G_\delta\|_{M_{0,2}^{\infty,1}}^2\hspace{-1pt}+\hspace{-1pt}\|G_\delta\|_{M^{\infty,1}}\|\partial_\delta G_\delta\|_{M^{\infty,1}}\right) .
\end{equation*}
Therefore, by Lemma \ref{lem:norm-bd-edges}, for $\delta_1\not=\delta_2$,
\begin{equation*}
	\frac{\big|\|p(T_{\delta_1})\|_{B(L^2)}\hspace{-1pt} - \hspace{-1pt}\|p(T_{\delta_2})\|_{B(L^2)}\big|}{|\delta_1\hspace{-1pt}-\hspace{-1pt}\delta_2|}\hspace{-1pt}\lesssim \hspace{-1pt}  \sup_{|\delta|<\delta_0}\hspace{-2pt}\left(\|G_\delta\|_{M_{0,2}^{\infty,1}}^2\hspace{-1pt}+\hspace{-1pt}\|G_\delta\|_{M^{\infty,1}}\|\partial_\delta G_\delta\|_{M^{\infty,1}}\right),
\end{equation*}
holds uniformly for all polynomials $p\in\mathcal{P}(T_0)$.
This shows that the number $C_{\mathcal{P}(T_0)}$ defined in Lemma \ref{lem:beckus-bellissard} satisfies
\begin{equation}\label{eq_ct}
	C_{\mathcal{P}(T_0)}\lesssim  \sup_{|t|<\delta_0}\left(\|G_t\|_{M_{0,2}^{\infty,1}}^2+\|G_t\|_{M^{\infty,1}}\|\partial_tG_t\|_{M^{\infty,1}}\right).
\end{equation}
In addition, reinspection of the previous estimates shows that $\{T_\delta\}_{|\delta|<\delta_0}$ is $(p2)$-Lipschitz continuous.
Hence, we can invoke Lemma \ref{lem:beckus-bellissard} and the conclusion follows from \eqref{eq_ct}.\pbox

\section{Gabor Frames}\label{sec:gabor}
We now apply the results on the Lipschitz continuity of
the spectral edges  to the Gabor frame operator.

Let $g\in M^1(\R^d)$ and $\Lambda\subset \R^{2d}$ be a relatively
separated set.  The frame operator of the associated set of
phase-space shifts is 
\[S_{g,\Lambda}f = \sum _{\lambda \in \Lambda } \langle f, \rho
(\lambda )g\rangle \rho (\lambda )g= \sum _{\lambda \in \Lambda }
q(\lambda ) f, \qquad f \in L^2(\mathbb{R}^d),\]
where $q$ is the rank-one projection \eqref{eq_q}. The Weyl symbol of $q(z)$ is just the shift $T_z\mathcal{W}(g)$ of the Wigner distribution of $g$ \eqref{eq_wigner}. Hence, the Weyl symbol of $S_{g,\Lambda }$ is
$$
\sigma_{g,\Lambda}=\sum_{\lambda\in\Lambda} T_\lambda \mathcal{W}(g).
$$
The spectral extreme values of $S_{g,\Lambda}$ and $S_{g,\alpha\Lambda}$ are
equal to the optimal frame bounds \eqref{eq:c1} of $\mathcal{G}(g,\Lambda)$ and
$\mathcal{G}(g,\alpha\Lambda)$ respectively. We set \[1/\alpha=\sqrt{1+\delta}.\]  The Weyl symbol corresponding to $S_{g,\alpha\Lambda}$ is
$$
\sigma _{g,\alpha \Lambda }=\sum_{\lambda\in\Lambda} T_{\alpha
  \lambda} \mathcal{W}(g) = D_{\sqrt{1+\delta }}\Big( \sum _{\lambda
  \in \Lambda } T_\lambda  D_{1/\sqrt{1+\delta}}\mathcal{W}(g)\Big)\, .
$$
Thus $\sigma _{g,\alpha \Lambda } = D_{\sqrt{1+\delta }} G_\delta $
with
$$
G_\delta = \sum _{\lambda
  \in \Lambda }   T_\lambda D_{1/\sqrt{1+\delta}}\mathcal{W}(g) \, .
$$
In order to apply Theorem~\ref{thm:main-general}, we need to calculate
the norms of $G_\delta$ and $\partial_\delta G_\delta$ in the corresponding weighted
Sj\"ostrand classes.

\begin{lemma} \label{sjosymbol}
Let $\Lambda\subset \R^{2d}$ be relatively separated and $0<\delta_0<1$. If $g\in M^1_2(\rd )$, then
\begin{enumerate}
\item[(i)] $\|  G_\delta\|_{M^{\infty,1}}\lesssim 
\emph{rel}(\Lambda)\cdot (1+\delta)^d  \cdot \|g\|_{M^1}^2$,$\qquad  \delta \in [0,\infty)$,
\item[(ii)] $
\|  G_\delta\|_{M^{\infty,1}_{0,2}}\lesssim \emph{rel}(\Lambda)\cdot  (1-\delta_0)^{-1}\cdot\|g\|_{M^1_2}^2$,$\qquad  \delta \in (-\delta_0,\delta_0)$,   
\item[(iii)]   $\|\partial_\delta G_\delta\|_{M^{\infty,1}}\lesssim \emph{rel}(\Lambda)\cdot  (1-\delta_0)^{-1}\cdot\|g\|_{M^1_2}^2$, $\qquad  \delta \in (-\delta_0,\delta_0)$.
\end{enumerate}
\end{lemma}

  \proof
  Let $\mu = \sum _{\lambda \in \Lambda } \delta _\lambda $. Then $G_\delta = \mu \ast D_{1/\sqrt{1+\delta }}\W
  (g)$ and  $\|\mu\|_{ M^\infty}\lesssim \text{rel}(\Lambda)$ by Lemma
  \ref{lemma_mu}. 
   
Furthermore, since $g\in M^1_2(\rd )$, its Wigner distribution satisfies
 \[\W(g) \in M^1_{0,2}(\rdd ),\] 
as a consequence of Lemma~\ref{lem:aux-W(G)}. The
convolution relation~\eqref{eq:c7} and the dilation
property~\eqref{eq_cx} on $\rdd$ show  that 
\begin{align*}
	\|G_\delta\|_{M^{\infty,1}} &\lesssim \| \mu\|_{M^\infty}\|
	D_{1/\sqrt{1+\delta }} \W (g)\|_{M^1} 
	\lesssim \text{rel}(\Lambda)\cdot (1+\delta)^d \cdot\|g\|_{M^1}^2,
\end{align*}
as claimed in (i). For (ii) we argue similarly:
\begin{align*}
  \|G_\delta\|_{M^{\infty ,1}_{0,2}} &\lesssim \| \mu\|_{M^\infty}\|
D_{1/\sqrt{1+\delta }} \W (g)\|_{M^1_{0,2}} 
\lesssim \max \big\{1,   ({1+\delta} )^{-1} \big\}
\| \mu\|_{M^\infty}\|g\|_{M^1_2}^2 \\
& \lesssim (1-\delta_0)^{-1}\cdot \text{rel}(\Lambda)\cdot\|g\|_{M^1_2}^2.
\end{align*}
It remains to determine  $\partial_\delta G_\delta$ and estimate
its norm.
First, we note
\begin{align*}
\partial_\delta G_\delta(z)&=\partial_{\delta}\left( \sum_{\lambda\in\Lambda}\mathcal{W}(g)\left(\frac{z-\lambda}{\sqrt{1+\delta}}\right)\right)
\\
&=  \sum_{\lambda\in\Lambda}\sum_{i=1}^{2d}-\frac{z_i-\lambda_i}{2(1+\delta)^{3/2}}\partial_i\mathcal{W}(g)\left(\frac{z-\lambda}{\sqrt{1+\delta}}\right)
\\
  &  =-\frac{1}{2(1+\delta) }\ \mu \ast
  D_{1/\sqrt{1+\delta}}\left(\sum_{i=1}^{2d}X_i\partial_i
  \mathcal{W}(g)  \right)(z)\, . 
\end{align*}
Using~\eqref{eq:c7} and~\eqref{eq_cx} as above, we prove (iii):
\begin{align*}
\|\partial_\delta G_\delta \|_{M^{\infty,1}}&\lesssim (1-\delta_0)^{-1}
\cdot\|\mu\|_{M^\infty}\sum_{i=1}^{2d}\| X_i \partial_i \mathcal{W}(g)\|_{M^1}
\\
&\lesssim (1-\delta_0)^{-1}
\cdot \text{rel}(\Lambda) \cdot \|\mathcal{W}(g)\|_{M_{1,1}^1}
\\
&\lesssim (1-\delta_0)^{-1} \cdot 
\text{rel}(\Lambda)
\cdot
\|g\|_{M_2^1}^2\,  .
\end{align*}
In the last step we have applied \eqref{eq:ju1} and Lemma \ref{lem:aux-W(G)}.
\pbox

An application of Theorem~\ref{thm:main-general} now allows us to show the Lipschitz continuity of the frame bounds of  $\mathcal{G}(g,\alpha\Lambda)$.

\medskip

\noindent \textbf{Proof of Theorem~\ref{blt}:} 
Recall that $\alpha^{-1}=\sqrt{1+\delta}$ with $\delta \in (-1,+\infty)$. Suppose first that $\delta \leq 1/2$ and set
$\delta_0=\max\{1/2,1-\alpha_0^2\}$. Then $0<\delta_0<1$. Let us check that $\delta \in (-\delta_0,\delta_0)$. By assumption, $\delta \leq 1/2 \leq \delta_0$. In addition,
$\sqrt{1+\delta}=1/\alpha > \alpha_0$, and consequently
$\delta > \alpha_0^2-1$, which shows that $-\delta < 1-\alpha_0^2 \leq \delta_0$. We now invoke Theorem~\ref{thm:main-general} and Lemma~\ref{sjosymbol}  to conclude that 
\begin{align*}
 \left|\sigma_\pm(S_{g,\Lambda})-\sigma_\pm(S_{g,\alpha\Lambda})\right|&\lesssim |\delta|\cdot(1-\delta_0)^{-(d+1)}\cdot\sup_{|t|<\delta_0}\left(\|G_t\|_{M^{\infty,1}_{0,2}}+\|\partial_tG_t\|_{M^{\infty,1}}\right) 
 \\
 &\lesssim |\delta|\cdot\text{rel}(\Lambda)\cdot (1-\delta_0)^{-(d+2) } \cdot\|g\|_{M^1_2}^2 
 \\
 &\leq |\delta|\cdot\text{rel}(\Lambda)\cdot  \alpha_0 ^{-2 (d+2) } \cdot\|g\|_{M^1_2}^2 .
\end{align*}
On the other hand, if $\delta \geq 1/2$, we use the following crude estimate based on Theorem \ref{lem:bdd}, \eqref{eq_c} and Lemma \ref{sjosymbol}:
\begin{align*}
\left|\sigma_\pm(S_{g,\Lambda})-\sigma_\pm(S_{g,\alpha\Lambda})\right| &\leq \left|\sigma_\pm(S_{g,\Lambda})\right|+\left|\sigma_\pm(S_{g,\alpha\Lambda})\right|
\\
&\leq \| S_{g,\Lambda} \|_{B(L^2)} + \| S_{g,\alpha\Lambda} \|_{B(L^2)}
\\
&\lesssim \| G_0 \|_{M^{\infty,1}}+
\| D_{\sqrt{1+\delta}} G_\delta \|_{M^{\infty,1}}
\\
&\lesssim (1+\delta)^d \big( \| G_0 \|_{M^{\infty,1}}+
\| G_\delta \|_{M^{\infty,1}}\big)
\\
&\lesssim (1+\delta)^{2d} \cdot\text{rel}(\Lambda)\cdot \|g\|^2_{M^1}
\\
&< |\delta| \cdot \alpha_0^{-4d} \cdot\text{rel}(\Lambda)\cdot \|g\|^2_{M^1}.
\end{align*}
Hence, for all $\delta \in (-1,\infty)$,
\begin{align}\label{eq_l}
	\left|\sigma_\pm(S_{g,\Lambda})-\sigma_\pm(S_{g,\alpha\Lambda})\right|&\lesssim |\delta|\cdot\text{rel}(\Lambda)\cdot  \alpha_0 ^{-4 d } \cdot\|g\|_{M^1_2}^2 .
\end{align}

Finally, observe that since $\alpha _0 < \alpha<1/\alpha_0$ and $\alpha_0<1$,
\begin{equation}\label{eq:est-alpha-delta}
|\delta|=\frac{|1-\alpha^2|}{\alpha^2}=\frac{1+\alpha}{\alpha^2}|1-\alpha|\leq
(\alpha _0^{-1} +  {\alpha _0}^{-2})  |1-\alpha| \leq
2 {\alpha _0}^{-2}    |1-\alpha|,
 \end{equation}
 which in combination with \eqref{eq_l} yields \eqref{eq_ll}.
\pbox

\begin{proof}[Proof of Theorem \ref{thm4}]
We let again $\alpha^{-1}=\sqrt{1+\delta}$ and take $|\alpha-1|<\varepsilon$ with $\varepsilon$ sufficiently small so that
$\delta \in (-\delta_0,\delta_0)$ and $\alpha_1 < \alpha < 1/\alpha_1$
with $\delta_0 \leq 1/2$ and $ 1/2 < \alpha_1 < 1$. We invoke Theorem \ref{thm:gap-bds}, Lemma \ref{sjosymbol} 
and \eqref{eq:est-alpha-delta} to obtain,
with possibly a smaller value of $\varepsilon$,
\begin{align*}
\big|\sigma^g_\pm(S_{g,{\Lambda}})-
\sigma^g_\pm(S_{g,{\alpha \Lambda}})\big| 
&\lesssim \frac{|\alpha-1|}{L(g)} \cdot 
\sup_{|t|<\delta_0}\left(\|G_t\|_{M^{\infty,1}}\|\partial_tG_t\|_{M^{\infty,1}}+\|G_t\|_{M^{\infty,1}_{0,2}}^2\right)
\\
&\lesssim \frac{|\alpha-1|}{L(g)} \cdot
\mathrm{rel}(\Lambda)^2\cdot \|g\|_{M^1_2}^4,
\end{align*}
as claimed.
\end{proof}


\begin{thebibliography}{10}

\bibitem{asfeka14}
G.~Ascensi, H.~G. Feichtinger, and N.~Kaiblinger.
\newblock Dilation of the {W}eyl symbol and {B}alian-{L}ow theorem.
\newblock {\em Trans. Amer. Math. Soc.}, 366(7):3865--3880, 2014.

\bibitem{atmapu10}
N.~Athmouni, M. M\u{a}ntoiu, and R.~Purice.
\newblock On the continuity of spectra for families of magnetic
  pseudodifferential operators.
\newblock {\em J. Math. Phys.}, 51:083517, 2010.

\bibitem{avmosi90}
J.~Avron, P.~H.~M. van Mouche, and B.~Simon.
\newblock On the measure of the spectrum for the almost {M}athieu operator.
\newblock {\em Commun. Math. Phys.}, 132:103--118, 1990.


\bibitem{beck-thesis} S.~Beckus. \newblock Spectral approximation of aperiodic Schr\"odinger operators. PhD thesis. Friedrich-Schiller-University Jena, 2016.

\bibitem{bebe16} S.~Beckus and J.~Bellissard. \newblock Continuity of the spectrum of a field
of self-adjoint operators. \newblock {\em Ann. Henri Poincar\'e}, 17:3425--3442, 2016.


\bibitem{bebeni18} S.~Beckus, J.~Bellissard, and G.~de~Nittis. \newblock Spectral continuity for aperiodic quantum systems I. General theory. \newblock {\em J. Funct. Anal.}, 275(11):2917--2977, 2018.


\bibitem{beta21} S. Beckus and A. Takase. \newblock Spectral estimates of dynamically-defined and
amenable operator families.  \newblock {\em Preprint, arXiv:2110.05763}, 2021.

\bibitem{bel94}
J.~Bellissard.
\newblock Lipshitz continuity of gap boundaries for {H}ofstadter-like spectra.
\newblock {\em Commun. Math. Phys.}, 160:599--613, 1994.

\bibitem{BO20}
A.~B{\'e}nyi and K.~A. Okoudjou.
\newblock {\em Modulation Spaces}.
\newblock Applied and Numerical Harmonic Analysis. Birkh{\"a}user/Springer, New
  York, 2020.

\bibitem{BGL10}
A.~Borichev, K.~Gröchenig, and Y.~Lyubarskii.
\newblock Frame constants of {G}abor frames near the critical density.
\newblock {\em J. Math. Pures Appl.}, 94(2):170--182, 2010.

\bibitem{cogroe03}
E.~Cordero and K.~Gr{\"o}chenig.
\newblock Time-frequency analysis of localization operators.
\newblock {\em J. Funct. Anal.}, 205(1):107--131, 2003.

\bibitem{coou12}
E.~Cordero and K.~A. Okoudjou.
\newblock Dilation properties for weighted modulation spaces.
\newblock {\em J. Funct. Spaces}, 2012, 2012.

\bibitem{CRbook}
E.~Cordero and L.~Rodino.
\newblock {\em Time-Frequency Analysis of Operators}.
\newblock De Gruyter Studies in Mathematics. De Gruyter, Berlin,
  2020.

\bibitem{co10}
H.~D. Cornean.
\newblock On the {L}ipschitz continuity of spectral bands of {H}arper-like and
  magnetic {S}chr\"odinger operators.
\newblock {\em Ann. Henri Poincar\'e}, 11:973--990, 2010.

\bibitem{cohepu21}
\newblock
H.~D. Cornean, B. Helffer and R.~Purice. \newblock
Spectral analysis near a Dirac type crossing in a weak non-constant magnetic field. \newblock {\em Trans. Amer. Math. Soc.} 374(10):7041--7104, 2021. 

\bibitem{copu12} H.~Cornean and R.~Purice. \newblock On the regularity of the Hausdorff distance between spectra of perturbed magnetic Hamiltonians. \newblock In: Spectral analysis of quantum Hamiltonians. Vol. 224. Oper. Theory Adv.
Appl., Birkh\"auser Basel,   2012.

\bibitem{copu15}
H.~D. Cornean and R.~Purice.
\newblock Spectral edge regularity of magnetic {H}amiltonians.
\newblock {\em J. London Math. Soc.}, 92(1):89--104, 2015.

\bibitem{ell82}
G.~Elliott.
\newblock Gaps in the spectrum of an almost periodic {S}chr\"odinger operator.
\newblock {\em C. R. Math. Rep. Acad. Sci. Canada}, 4(5):255--260, 1982.

\bibitem{fei83}
H.~G. Feichtinger.
\newblock Banach convolution algebras of {W}iener type.
\newblock In {\em Functions, Series, Operators}, volume I, II, pages 509--524.
  North-Holland, Amsterdam, 1983.

\bibitem{FK04}
H.~G. Feichtinger and N.~Kaiblinger.
\newblock Varying the time-frequency lattice of {G}abor frame.
\newblock {\em Trans. Amer. Math. Soc.}, 356(5):2001--2023 (electronic), 2004.

\bibitem{fo89}
G.~B. Folland.
\newblock {\em Harmonic {A}nalysis in {P}hase {S}pace}.
\newblock Princeton University Press, Princeton, 1989.

\bibitem{gesha22} M.~Gerhold and O.~M.~Shalit. \newblock Dilations of $q$-commuting unitaries. \newblock
{\em Int. Math. Res. Not.}, 2022(1):63--88, 2022.

\bibitem{groe1}
K.~Gr{\"o}chenig.
\newblock {\em {F}oundations of {T}ime-{F}requency {A}nalysis}.
\newblock Applied and Numerical Harmonic Analysis. {B}irkh{\"a}user {B}oston, 2001.

\bibitem{GR06}
K.~Gr{\"o}chenig.
\newblock Composition and spectral invariance of pseudodifferential operators
  on modulation spaces.
\newblock {\em J. Anal. Math.}, 98:65--82, 2006.

\bibitem{Gro06}
K.~Gr{\"o}chenig.
\newblock Time-frequency analysis of {S}j{\"o}strand’s class.
\newblock {\em Revista Mat. Iberoam.}, 2(2):703--724, 2006.

\bibitem{grorcero15}
K.~Gr{\"o}chenig, J.~Ortega-Cerd{\'a}, and J.~L. Romero.
\newblock Deformation of {G}abor systems.
\newblock {\em Adv. Math.}, 227:388--425, 2015.

\bibitem{heil07}
C.~Heil.
\newblock History and evolution of the density theorem for {G}abor frames.
\newblock {\em J. Fourier Anal. Appl.}, 13(2):113--166, 2007.

\bibitem{KS14}
T.~Kloos and J.~St{\"o}ckler.
\newblock Zak transforms and {G}abor frames of totally positive functions
  and exponential {B}-splines.
\newblock {\em J. Approx. Theory}, 184:209–237, 2014.

\bibitem{ko03}
M.~Kotani.
\newblock Lipschitz continuity of the spectra of the magnetic transititon
  operators on a crystal lattice.
\newblock {\em J. Geom. Phys.}, 47:323--342, 2003.

\bibitem{Sjo94}
J.~Sj{\"o}strand.
\newblock An algebra of pseudodifferential operators.
\newblock {\em Math. Res. Lett.}, 1(2):185--192, 1994.

\bibitem{Sjo95}
J.~Sj{\"o}strand.
\newblock Wiener type algebras of pseudodifferential operators.
\newblock In {\em S{\'e}minaire sur les {\'E}quations aux D{\'e}riv{\'e}es
  Partielles (Polytechnique)}. 1994–1995, pages 1--19, Exp. No. IV, 21. École
  Polytech., Palaiseau, 1995.

\bibitem{suto07}
M.~Sugimoto and N.~Tomita.
\newblock The dilation property of modulation spaces and their inclusion
  relation with {B}esov spaces.
\newblock {\em J. Funct. Anal.}, 248(1):79--106, 2007.

\bibitem{tof04}
J.~Toft.
\newblock Continuity properties for modulation spaces, with applications to
  pseudo-differential calculus {I}.
\newblock {\em J. Funct. Anal.}, 207(2):399--429, 2004.

\end{thebibliography}
\end{document}